\documentclass{siamltex}
\usepackage[english]{babel}
\usepackage[utf8]{inputenc}

\usepackage{times}
\usepackage[T1]{fontenc}

\usepackage{epsfig}
\usepackage{graphics,tikz}
\usepackage{color}
\usepackage{amsmath,amssymb,psfrag,mathrsfs,commath}
\usepackage{amsfonts,stmaryrd}
\usepackage[normalem]{ulem}
\usepackage{algorithm,algorithmicx,algpseudocode}

\usepackage{subfigure,epsfig,graphicx,psfrag,url,fancyvrb}
\usepackage[bbgreekl]{mathbbol}

\usepackage[colorlinks=true,bookmarks=false,hypertexnames=false]{hyperref} 
\usepackage{cleveref}

\usepackage{float}
\usepackage{comment}
\usepackage{xcolor}
\usepackage{subeqnarray}
\usepackage{tikz}
\usetikzlibrary{shapes.multipart,calc,fit}
\usetikzlibrary{matrix}

\newcommand{\reals}{\mathbb{R}}
\newcommand{\irange}[1][k]{[#1]}
\newcommand{\grp}{\mathcal{G}}

\newcommand{\gre}{\mathcal{E}}

\DeclareMathOperator*{\argmin}{arg\,min}

\makeatletter
\newcommand{\makesmalleroperator}[1]{%
  \expandafter\let\csname saved\string#1\endcsname#1
  \def#1{\mathop{\small@mathchoice#1}}
}
\def\small@mathchoice#1{%
  \mathchoice{\textstyle\@nameuse{saved\string#1}}%
             {\@nameuse{saved\string#1}}%
             {\@nameuse{saved\string#1}}%
             {\@nameuse{saved\string#1}}%
}
\makeatother

\makeatletter
\def\bbordermatrix#1{\begingroup \m@th
  \@tempdima 4.75\p@
  \setbox\z@\vbox{%
    \def\cr{\crcr\noalign{\kern2\p@\global\let\cr\endline}}%
    \ialign{$##$\hfil\kern2\p@\kern\@tempdima&\thinspace\hfil$##$\hfil
      &&\quad\hfil$##$\hfil\crcr
      \omit\strut\hfil\crcr\noalign{\kern-\baselineskip}%
      #1\crcr\omit\strut\cr}}%
  \setbox\tw@\vbox{\unvcopy\z@\global\setbox\@ne\lastbox}%
  \setbox\tw@\hbox{\unhbox\@ne\unskip\global\setbox\@ne\lastbox}%
  \setbox\tw@\hbox{$\kern\wd\@ne\kern-\@tempdima\left.\kern-\wd\@ne
    \global\setbox\@ne\vbox{\box\@ne\kern2\p@}%
    \vcenter{\kern-\ht\@ne\unvbox\z@\kern-\baselineskip}\,\right.$}%
  \null\;\vbox{\kern\ht\@ne\box\tw@}\endgroup}
\makeatother

\makesmalleroperator\bigotimes

\usepackage{helvet}
\usepackage{lineno}
\newtheorem{remark}[theorem]{Remark}
\usepackage{booktabs,multirow}
\DeclareMathOperator*{\trace}{tr}
\graphicspath{{./figures/}}
\newcommand{\dev}[1]{\textcolor{red}{\texttt{#1}}}

\newcommand{\revv}[1]{#1}
\newcommand{\devv}[1]{}
\newcommand{\dele}[1]{}

\newcommand{\comm}[1]{} 
\newcommand{\revf}[1]{#1}

\newcommand{\rev}[1]{#1}
\newcommand{\revs}[1]{#1}
\newcommand{\revn}[1]{#1}

\usepackage{marginnote}
\newcommand{\margn}[1]{}
\newcommand{\mstar}{M^{\star}}
\newcommand{\rowof}[2]{#1_{#2,:}}

\newcommand{\greg}{GraphReg}
\newcommand{\freg}{FrobReg}
\newcommand{\noreg}{Unregularized} 

\usepackage{xcolor}
\usepackage{pifont}
\newcommand{\cmark}{\ding{51}}%
\newcommand{\bcheck}{{\color{black}\cmark}}

\title{Alternating minimization algorithms for graph regularized tensor completion}
\author{Yu Guan\thanks{Huawei European Research Institute (ricky7guanyu@gmail.com)}
\and Shuyu Dong\thanks{LISN, INRIA, Universit\'e Paris-Saclay (shuyu.dong@inria.fr)}
\and Bin Gao\thanks{LSEC, Academy of Mathematics and Systems Science, Chinese Academy of Sciences (gaobin@lsec.ac.cc.cn)}
\and P.-A. Absil\thanks{Department of INMA, Universit\'e catholique de Louvain (pa.absil@uclouvain.be)}
\and François Glineur\thanks{Department of INMA, Universit\'e catholique de Louvain (francois.glineur@uclouvain.be)}}

\begin{document}

\maketitle


\begin{abstract}
{\it We consider a Canonical Polyadic (CP) decomposition approach to low-rank tensor completion (LRTC) by incorporating external pairwise similarity relations through graph Laplacian regularization on the CP factor matrices. {The usage of graph regularization entails benefits in the learning accuracy of LRTC, but at the same time, induces coupling graph Laplacian terms that hinder the optimization of the tensor completion model.} 
{In order to solve graph-regularized LRTC, we propose efficient alternating minimization algorithms by leveraging the block structure of the underlying CP decomposition-based model. For the subproblems of alternating minimization, a linear conjugate gradient subroutine is specifically adapted to graph-regularized LRTC. 
Alternatively, we circumvent the complicating coupling effects of graph Laplacian terms by using an alternating directions method of multipliers.} 
Based on the Kurdyka-{\L}ojasiewicz property, we show that the sequence generated by the proposed algorithms globally converges to a critical point of the objective function. Moreover, the complexity and convergence rate are also derived. 
In addition, numerical experiments including synthetic data and real data {show that the graph regularized tensor completion model has improved recovery results compared to those without graph regularization, and that the proposed algorithms achieve gains in time efficiency over existing algorithms.} 
}
\end{abstract}



\section{Introduction}\label{sec:introduction}
\rev{Matrix and tensor (also known as multidimensional arrays) completion arise in many areas such as signal processing for EEG data~\cite{morup2006parallel} and MRI (magnetic resonance imaging)~\cite{banco2016sampling}, \revn{genetic} data analysis \cite{li2021imputation} and 
image and video restoration~\cite{bertalmio2000image,liu2012tensor}. 
In these applications, the data matrix or data tensor is often only partially observed, undersampled, or sampled with noise; matrix or tensor completion is an abstraction of the problem of recovering such data. While it is unlikely to recover the missing data if the hidden matrix or tensor is unstructured, it is shown that matrix completion can indeed be solved when the hidden matrix has a low rank~\cite{candes2009exact,recht2010guaranteed}, 
and instead of the matrix rank, the matrix nuclear norm is a convex relaxation that guarantees exact solutions to the low-rank matrix completion problem~\cite{Mazumder2010}.} %
Generalizing from the matrix case to the tensor case, 
several works~\cite{gandy2011tensor,liu2012tensor} 
extended the matrix nuclear norm-based regularization to the tensor completion problem. Liu et al.~\cite{liu2009tensor} introduced an extension of the matrix nuclear norm to the low-rank tensor completion (LRTC) problem and later defined the nuclear norm of a tensor as a convex combination of nuclear norms of its unfolding matrices~\cite{liu2012tensor}. Given an $k$-th order tensor $\mathcal{T}\in\mathbb{R}^{m_{1}\times\ldots\times m_{k}}$, the nuclear norm-based tensor completion model is as follows, 
\begin{equation}\label{origin}
\min_{\mathcal{Z}\in\mathbb{R}^{m_{1}\times\ldots\times m_{k}}} \quad\frac{1}{2}\|\mathcal{P}_{\Omega}(\mathcal{T}-\mathcal{Z})\|_{\mathrm{F}}^{2}+\sum_{i=1}^{k}\lambda_{i}\|\mathcal{Z}_{(i)}\|_{*},
\end{equation}
where $\mathcal{P}_{\Omega}$ is the projection operator that only retains the revealed entries of $\mathcal{T}$, recorded in the index set $\Omega\subset\irange[m_1]\times\cdots\times\irange[m_k]$, and $\|\mathcal{Z}_{(i)}\|_*$ denotes the matrix nuclear norm of the mode-$i$ matricization %
(\Cref{sec:preliminaries}) of $\mathcal{Z}$. The nuclear norm terms~$\|\mathcal{Z}_{(i)}\|_{*}$ in~\eqref{origin} \rev{promote solutions of $\mathcal{Z}$ with low-rank matricizations}. 
\rev{However, the model~\eqref{origin} has a memory requirement of $O(m_1\dots m_k)$, and the nuclear norm term in~\eqref{origin} involves subdifferential computations that require singular value decomposition of the unfolding matrices~$\mathcal{Z}_{(i)}$, 
which \revn{can be} very large since their sizes $(m_i\times \Pi_{j\neq i}m_j$) grow quickly with the tensor size $(m_1,\dots,m_k)$}. 
Therefore, \rev{many LRTC approaches use low-rank tensor decompositions such as 
the Tucker decomposition and the Canonical Polyadic (CP) decomposition %
to systematically limit the number of parameters of the tensor model~\cite{jain2014provable,zhou2017tensor,nimishakavi2018dual,lacroix2018canonical,dong2022new}. 
Other decomposition-based approaches to the LRTC problem include hierarchical tensor representations~\cite{da2013hierarchical,rauhut2015tensor}, 
tensor train decomposition~\cite{steinlechner2016riemannian,cai2022tensor} and tensor ring decomposition~\cite{wang2017efficient,gao2023riemannian}.}  

\rev{In addition to low-rankness, auxiliary similarity information about the inter-relations between data entries, also called side information, %
is also an important source for refining the solutions of data recovery problems such as matrix and tensor completion. 
For matrix completion, auxiliary information is used in the form of graph regularization \cite{srebro2010collaborative,zhou2012kernelized,rao2015collaborative}, which is typically in the following form~\cite{rao2015collaborative}, for a partially observed matrix $M\in\mathbb{R}^{m_1\times m_2},$} 
\begin{equation}\label{rao}
	\min_{X, Y}\|\mathcal{P}_{\Omega}(M - XY^{\top})\|_{\mathrm{F}}^{2}+\trace(X^{\top}\mathcal{L}^{(1)} X) + \trace(Y^{\top} \mathcal{L}^{(2)} Y) 
\end{equation}
\rev{where $\mathcal{L}^{(1)}$ is %
a graph Laplacian matrix of a graph encoding certain row-wise similarities of the data entries in $M$, and $\mathcal{L}^{(2)}$, likewise, is a graph Laplacian matrix for the column-wise similarities.}
\rev{Depending on the specific domain of application, the graph structures needed for graph regularization are either directly related to the problem background, such as traffic network of highways for road traffic prediction \cite{lan2022dstagnn}, or they can be constructed from auxiliary information, such as user communities or item similarities for online recommendation tasks~\cite{rao2015collaborative,dong2021riemannian}.} 
\rev{The usage of the graph Laplacian-based functions in \eqref{rao} is motivated by the following reason. Given an undirected graph $\grp^{(1)}$ on the row index set of $X$ and a weighted graph adjacency matrix $W\in\reals^{m_1\times m_1}$ associated with $\grp^{(1)}$, the Laplacian matrix of $(\grp^{(1)}, W)$, defined as  
$\mathcal{L}^{(1)} = \text{diag}(W\mathbf{1}) - W$, 
has the following property~\cite[Section 1.4]{chung1997spectral}, %
\begin{align}\label{eq:flf}
\trace (X^{\top}\mathcal{L}^{(1)} X) 
= \sum_{i=1}^{m_1}\sum_{j=1}^{m_1}W_{ij}\|X_{i,:}-X_{j,:}\|_{2}^{2}.
\end{align}
From the form of the weighted sum of squared row-wise differences in \eqref{eq:flf}, one can see that minimizing~\eqref{eq:flf}, subject to a constraint on $X$ for a certain data fitting objective, promotes solutions that are piecewise smooth on the graph $\grp^{(1)}$\revn{~\cite{shuman2013emerging}}.} 
\rev{More generally, the graph Laplacian-based function \eqref{eq:flf} is a ubiquitous tool in data analysis~\cite{coifman2005geometric}, graph signal processing~\cite{shuman2013emerging}, 
semi-supervised learning~\cite{ando2006learning}, and image restoration~\cite{pang2017graph,cheung2018graph,zeng2019deep}. 
} 

\rev{In the context of tensor completion, 
Narita et al. \cite{narita2012tensor} considered using auxiliary graph information to regularize tensor completion solutions. They proposed tensor factorization models involving graph Laplacian regularizers for the completion of third-order tensors. 
The graph Laplacian regularization in \cite{narita2012tensor} is designed in two ways: the first way called {\it cross-mode} is through a graph Laplacian-based norm of the Kronecker product of the three tensor factor matrices, and the second way called {\it within-mode} is through the sum of three graph Laplacian-based norms of the respective factor matrices. 
Other related works on tensor completion are summarized in \Cref{tab:relatedwork-b}, which will be further explained in \Cref{ssec:priorwork} after the necessary notation for tensors is introduced. 
}

\begin{table}[!h]
\footnotesize
\centering
\caption{Comparison with related work on tensor completion. 
The column `$k$ (order)' refers to applicability to $k$-th order tensors. The column `Cost' refers to the dominant per-iteration cost. The column `Convergence' refers to convergence result for the proposed method. `--' means not available, and the check mark means the contrary. NCG refers to the nonlinear conjugate gradient method. NMF refers to nonnegative matrix factorization, and `*' means additional nonnegative constraints.} 
\label{tab:relatedwork-b}
\begin{tabular}{c|ccrlcc} 
\hline
                                & Model type              & Optimization       & $k$ (order) & \multicolumn{1}{c}{{Cost}}                        & Graph reg. & Convergence  \\ 
\hline
INDAFAC~\cite{tomasi2005parafac}    & \eqref{prog:lrtc-a} & Gauss-Newton      & $3$     & $O((m_1m_2m_3)^3)$              & --         & --            \\ 
CP-WOPT~\cite{acar2011scalable}     & \eqref{prog:lrtc-a} & NCG & $\geq3$   &
\multicolumn{1}{c}{--}   & --   & -- \\
BPTF~\cite{xiong2010temporal}       & \eqref{prog:lrtc-a} & Bayesian, MCMC    & $3$     & $O(|\Omega|R^2)$                & --         & --            \\
TNCP~\cite{liu2014trace}            & \eqref{prog:lrtc-b} & ADMM              & $\geq3$ & $O((k+1)R\Pi_{i=1}^k m_i)$      & --         & \bcheck       \\
TFAI~\cite{narita2012tensor}        & \eqref{prog:lrtc-b} & EM-like, NCG      & $3$     & $O((|\Omega|+\text{nnz}(L))R)$  & \bcheck    & --            \\
AirCP~\cite{ge2016uncovering}       & \eqref{prog:lrtc-b} & ADMM              & $3$     & \multicolumn{1}{c}{--}                              & \bcheck    & \bcheck       \\
\rev{FIST~\cite{li2021imputation}} &~~\eqref{prog:lrtc-a}* & NMF & $3$ & $O((|\Omega|+\sum_{i=1}^{3}m_i^2))R)$ & \bcheck & -- \\ 
Ours                                & \eqref{prog:lrtc-a} & AltMin, ADMM      & $\geq3$ & $O((|\Omega|+\text{nnz}(L))R)$  & \bcheck    & \bcheck       \\ 
\hline
\end{tabular}
\end{table}
\paragraph*{\bf Contribution} 
In this paper, \rev{we address the tensor completion problem for $k$-th order tensors by considering a CP decomposition model~\eqref{problem} that incorporates auxiliary information via graph Laplacian regularization in the form of \eqref{eq:flf}.} 
\rev{The underlying CP decomposition, similar to the matrix problem~\eqref{rao} and the graph regularized problem with third-order tensors~\cite{narita2012tensor}, has a block structure such that the search of full $(m_1,\dots, m_k)$-tensors can be divided into a sequence of $k$ smaller subproblems on the CP factor spaces. 
Taking the graph Laplacian regularizers into account, each subproblem is formulated explicitly into a least-squares problem.
Then we propose an alternating minimization   algorithm (AltMin) for optimizing the graph-regularized tensor completion model. 
An efficient Hessian-vector multiplication scheme is adapted to a linear conjugate gradient (CG) method for solving the subproblems in the alternating minimization procedure.}  
We provide a proof for the convergence of iterates of the proposed AltMin algorithm to a critical point of the objective function according to the Kurdyka-{\L}ojasiewicz (K{\L}) property. 

\rev{Notice that the graph Laplacian regularization induces a coupling effect on the Hessian coefficients of each subproblem in AltMin, which complicates the resolution of the underlying least-squares problem. Therefore we consider variable splitting for the graph-regularized tensor completion model alternatively. More precisely, by splitting the CP factors in the data fitting term and the graph regularization term, we propose an alternating direction method of multipliers (ADMM), which decouples the Hessian coefficients into two splitting subproblems. 
One of the resulting subproblems involves a graph Laplacian term while the other does not, hence the latter one can be solved in parallel.}

We conduct tensor completion experiments %
on both synthetic and real data and show that the proposed algorithms \rev{entail improved recovery results by using appropriate auxiliary graph information}. 
\rev{The graph Laplacian regularization shows significant improvement on unregularized and nuclear norm-regularized models,} especially when the fraction of revealed data is small. 
\rev{The two proposed algorithms also show speedups over several baseline methods on the synthetic and real data experiments.}

\paragraph*{\bf Organization} 
We organize this paper as follows. We begin with the introduction of notations and definitions in Section~\ref{sec:preliminaries}, and \rev{we introduce the LRTC model with a graph Laplacian-based regularizer and then present several related work for CP decomposition-based tensor completion including existing methods using graph regularization.} 
In Section \ref{sec:algorithm}, an alternating minimization algorithm using linear CG for solving the subproblems is proposed; an ADMM algorithm is also developed for solving the graph-regularized LRTC problem.  
Convergence analysis of the AltMin algorithm is given in
Section~\ref{sec:convergence}. Numerical experiments together with some
interesting observations are presented in Section~\ref{sec:numerical}. Conclusion is shown in Section \ref{sec:conclusion}.

\section{Preliminaries and problem setting}\label{sec:preliminaries} 
In this section, we introduce the definition and notation of some tensor operations, and \revs{set up the target problem}. 
\revs{A real-valued order-$k$ tensor  
$\mathcal{Z} \in {\mathbb{R}^{m_{1}\times m_{2}\times\ldots\times m_{k}}}$ is a \revn{multiway} array in which each entry, denoted as $\mathcal{Z}_{\ell_{1},\ldots,\ell_{k}}$, is accessed via $k$ indices $(\ell_{1},\ldots,\ell_{k}) \in [m_1]\times\dots\times [m_k]$, 
where $[m_i]$ denotes the index set $\{1,\ldots,m_i\}$ for the integer $m_i$. 
The vector of all \revn{ones} is denoted as $\mathbf{1}$ and the vector with one on the $i$-th \revn{entry} and \revn{zeros} elsewhere is denoted as $\mathbf{e}_i$. An $m\times m$ identity matrix is denoted as $I_{m}$. 
}

\rev{For $k\geq 2$, the outer product of $k$ vectors $\{u^{(1)},\dots, u^{(k)} \}$ is \revn{a $k$-th} order tensor, denoted as $\mathcal{Z}:=u^{(1)}\circ \dots \circ u^{(k)}$, such that $\mathcal{Z}_{\ell_1,\dots,\ell_k} = u^{(1)}_{\ell_1} \dots 
 u^{(k)}_{\ell_k}$.}  
The {\it Kronecker product} of two vectors 
\revs{$u\in\mathbb{R}^{m_{1}}$ and $v\in\mathbb{R}^{m_{2}}$} 
results in a vector 
\revs{${u}\otimes {v}\in\mathbb{R}^{m_{1}m_{2}}$} defined as
\revs{${u}\otimes {v} =(u_{1}{v}^{\top}, u_{2}{v}^{\top} \cdots u_{m_{1}}{v}^{\top})^{\top}$.} 
More compactly, we have 
\revs{$({u}\otimes{v})_{m_{2}(i-1)+j}=u_{i}v_{j}$ for $(i,j)\in [m_1]\times [m_2]$.}
\rev{The Kronecker product of two matrices $A\in\mathbb{R}^{m\times n}$ and $B\in\mathbb{R}^{p\times q}$ is the $pm\times qn$ matrix 
$$
A\otimes B = \begin{pmatrix} a_{11}B & \dots & a_{1n}B \\
\vdots & \ddots & \vdots \\
a_{m1}B & \dots & a_{mn}B \end{pmatrix}.
$$ 
The {\it Khatri-Rao product} $U \odot V$ of 
two matrices $U\in \mathbb{R}^{m_{1} \times R}$ and $V\in \mathbb{R}^{m_{2}\times R}$ with the same number of columns is a matrix of size $m_{1}m_{2} \times R$ whose $r$-th column is $U_{:,r} \otimes V_{:,r}$.} 

For a tensor
$\mathcal{Z}\in {\mathbb{R}^{m_{1}\times m_{2}\times\ldots\times m_{k}}}$, 
the {\it mode-$i$ matricization} $\mathcal{Z}_{(i)}$ is the $m_{i}\times (\prod_{j\neq i}m_{j})$ unfolding \rev{matrix} of $\mathcal{Z}$ along its $i$-th mode. 
\revs{The matricization $\mathcal{Z}_{(i)}$ satisfies the following identification of the matrix entry with the \revn{tensor entry:} $(\mathcal{Z}_{(i)})_{\ell_{i}, r_{i}} = \mathcal{Z}_{\ell_{1},\ldots ,\ell_{k}}$ where} 
\begin{equation}\label{eq:ind-mati}
r_{i}=1+\sum^{k}_{\substack{n=1\\ n \neq i}}(\ell_{n}-1)I_{n}, \quad \text{with~} \quad I_{n}=\prod^{n-1}_{\substack{j=1\\ j \neq i}}m_{j}.
\end{equation}

\revn{A} {\it Canonical Polyadic (CP) decomposition}~\cite{hitchcock1927expression,kruskal1977three,Kolda2009} %
of a tensor $\mathcal{Z}\in\reals^{m_1\times\cdots\times m_k}$ is defined and denoted as
\begin{equation}\label{cp_decomposition}
\revs{\mathcal{Z}=\llbracket U^{(1)},\ldots,U^{(k)}\rrbracket
=\sum_{r=1}^{R}U_{:,r}^{(1)}\circ \dots \circ U_{:,r}^{(k)}}, 
\end{equation}
\revs{where $U^{(i)}\in\mathbb{R}^{m_{i}\times R}$ for $i=1,\ldots,k$ and 
$R\in\mathbb{Z}$ are CP factor matrices and a rank parameter respectively satisfying \eqref{cp_decomposition}.} %
An equivalent CP form can be written as
\begin{align} \label{eq:unfold}
\mathcal{Z}_{(i)} = U^{(i)}(U^{(k)}\odot \ldots \odot U^{(i+1)} \odot U^{(i-1)}\odot \ldots \odot U^{(1)})^{\top}
\revs{=U^{(i)}((U^{(j)})^{\odot_{j\neq i}})^{\top},} 
\end{align}
where $\mathcal{Z}_{(i)}$ is the mode-$i$ tensor matricization.
\revs{Given an index $i \in [k]$, the subscript or superscript $(-i)$ means that the notion in question depends on the subset (or filtering) $[k]\backslash \{i\}$. Hence the following notations, $m_{(-i)}:=\prod_{j\neq i}m_j$ and $U^{(-i)}:= (U^{(j)})^{\odot_{j\neq i}}\in\reals^{m_{(-i)}\times R}$, are used for brevity whenever needed. 
}

\rev{The Frobenius inner product of two matrices $A,B$ of the same size, by definition, is denoted by $\langle A, B\rangle = \trace(A^{\top}B)$ interchangeably.} 
The {\it tensor inner product} of two tensors $\mathcal{Z}^{(1)},\mathcal{Z}^{(2)}\in {\mathbb{R}^{m_{1}\times m_{2}\times\ldots\times m_{k}}}$ %
of the same size is defined as 
\begin{align*}
    \langle \mathcal{Z}^{(1)},\mathcal{Z}^{(2)} \rangle=\sum_{\ell_{1}=1}^{m_{1}}\cdots \sum_{\ell_{k}=1}^{m_{k}}\mathcal{Z}^{(1)}_{\ell_{1}\ldots \ell_{k}}\mathcal{Z}^{(2)}_{\ell_{1}\ldots \ell_{k}}.
\end{align*}
The Frobenius norm of a tensor $\mathcal{Z}\in {\mathbb{R}^{m_{1}\times m_{2}\times\ldots\times m_{k}}}$, \rev{same as the matrix notation, is defined and denoted as 
$\|\mathcal{Z}\|_{\mathrm{F}}=\sqrt{ \langle \mathcal{Z},\mathcal{Z} \rangle}$.}

\subsection{Problem setting}\label{sec:problem}
We introduce our LRTC model 
in the form of CP decomposition %
with a graph Laplacian-based regularization, 
\revs{
\begin{align}\label{problem}
    \underset{U^{(1)},\ldots,U^{(k)}}{\text{min}} 
\frac{1}{2}\|\mathcal{P}_{\Omega}(\mathcal{T}-\llbracket U^{(1)},\cdots, U^{(k)}\rrbracket)\|_{\mathrm{F}}^{2}
+ \sum_{i=1}^{k}\frac{{\lambda_{i}}}{2}\trace\big((U^{(i)})^{\top} L^{(i)} U^{(i)}\big) 
+ \sum_{i=1}^{k}\frac{\lambda_{i}}{2}\|(U^{(j)})^{\odot_{j\neq i}}\|_{\mathrm{F}}^{2}
\end{align}
where 
$U:=(U^{(1)},\ldots,U^{(k)}) \in\mathbb{R}^{m_{1}\times R}\times \cdots \times \mathbb{R}^{m_{k}\times R}$ for a (strictly positive) rank parameter $R$,} 
\revs{$\Omega$ denotes the index set of} the revealed tensor entries
($\Omega\subset [m_1]\times\cdots\times[m_k]$) and $\mathcal{T}$ is the tensor
that is revealed only on~$\Omega$. The proportion of the \revn{revealed entries} 
$\frac{|\Omega|}{m_1...m_k}$ %
is referred to as the \emph{sampling rate}. %
The shifted graph Laplacian $L^{(i)}$ is defined as  
\begin{align}
L^{(i)}=\lambda_{L}\mathcal{L}^{(i)}+I_{m_{i}} \label{eq:def-Li} \quad\text{with}\quad
\mathcal{L}^{(i)}= \text{diag}(W^{(i)}\mathbf{1}) - W^{(i)}, 
\end{align} 
where ${W}^{(i)}$ is a \revs{weighted adjacency matrix of a given graph $\grp^{(i)}$ on the index set $[m_i]$, i.e., $W^{(i)}_{\ell_1,\ell_2} \neq 0$ if and only if $(\ell_1,\ell_2)$ is an (undirected) edge of $\grp^{(i)}$.}  
\revn{The Frobenius norm-based terms in the regularizer of~\eqref{problem} (when $\lambda_i >0$ and $\lambda_L = 0$) are related to the nuclear norm of the matricizations of $\mathcal{Z}$ 
through the following characterization~\cite{srebro2005maximum}: 
\begin{align}
    \label{eq:nuclearnorm-chara}
	 \|\mathcal{Z}_{(i)}\|_{*}= \min_{U^{(i)}\big((U^{(j)})^{\odot_{j\neq i}}\big)^{\top}=\mathcal{Z}_{(i)}}\frac{1}{2}\{\|U^{(i)}\|_{\mathrm{F}}^{2}+\|(U^{(j)})^{\odot_{j\neq i}}\|_{\mathrm{F}}^{2}\}.
\end{align}
In the general case ($\lambda_i>0$ and $\lambda_L>0$), the regularizer of~\eqref{problem} can be seen as a generalized form of the nuclear norm~\eqref{eq:nuclearnorm-chara}, which entails a generalization of the tensor nuclear norm.} 

The parameters $\lambda_{i}$ and $\lambda_{L}$ control the
trade-off between 
the regularization term and 
the data fitting term. 
\rev{More specifically, the graph Laplacian-based regularization term, as explained in %
\eqref{eq:flf}, 
has the effect of favoring solutions that tend to be piecewise smooth with respect to the graph links in $\grp^{(i)}$. 
Therefore, the regularization of \eqref{problem} (for $\lambda_L>0$) induces a
trade-off between data fitting on the revealed entries and piecewise smoothness
according to the given graph across the whole index set.} 
In particular, when $\lambda_{L} =0$, problem~\eqref{problem} reduces to a
Frobenius norm-regularized tensor completion model (for $\lambda_i>0$) or an
unregularized model (for $\lambda_i=0$).

\subsection{\rev{Related work}} 
\label{ssec:priorwork}
Among the prior work on LRTC using CP decomposition, 
we give a representative list of methods, including ours, \rev{that deal with either of the following two tensor completion models (while some methods are limited to third-order tensors): 
\begin{align}
    & \min_{U}\|P_{\Omega}(\llbracket U^{(1)}, \dots, U^{(k)}\rrbracket-\revn{\mathcal{T}})\|_{\mathrm{F}}^{2} + \psi(U);  \label{prog:lrtc-a} \\ 
    & \min_{\revn{\hat{\mathcal{T}}},U}\|\revn{\hat{\mathcal{T}}}-\llbracket U^{(1)}, \dots, U^{(k)}\rrbracket\|_{\mathrm{F}}^{2}+\psi(U)\quad \text{subject to}\quad \revn{\mathcal{P}_{\Omega}(\hat{\mathcal{T}})=\mathcal{P}_{\Omega}(\mathcal{T})},  \label{prog:lrtc-b}  
\end{align}
where $\psi$ is a (possibly zero-valued) regularizer. 
}

\rev{\revn{A comparison} with existing methods is given in~\Cref{tab:relatedwork-b}. Our work considers \eqref{problem} with general $k$-th order tensors, which is of type \eqref{prog:lrtc-a} where the graph Laplacian-based regularizer~\eqref{eq:flf} is enabled. The proposed algorithms (AltMin and ADMM) have per-iteration costs less than or comparable to others, and the convergence property of AltMin is given.} 
INDAFAC is a damped Gauss-Newton method proposed by Tomasi and
Bro~\cite{tomasi2005parafac} for solving %
\rev{\eqref{prog:lrtc-a} with $\psi=0$.} 
CP-WOPT is an algorithm by Acar et al.~\cite{acar2011scalable}  
for solving %
\rev{\eqref{prog:lrtc-a} with $\psi=0$,} 
\revv{and is available} in the Tensor Toolbox~\cite{bader2012matlab}. 
BPTF is a Bayesian probabilistic tensor CPD
algorithm by Xiong et al.~\cite{xiong2010temporal} 
\rev{for solving \eqref{prog:lrtc-a}} 
where the regularizer $\psi$ is composed of Frobenius norms of the factors of
$U$ and an $\ell_2$ norm-based function imposing columnwise smoothness of
$U^{(3)}$. 
TFAI is an algorithm \revn{for optimizing} the auxiliary-information model of Narita et al.~\cite{narita2012tensor}, 
\rev{which corresponds to \eqref{prog:lrtc-b}} 
where $\psi$ is a graph Laplacian-based regularizer (in the within-mode) as in~\eqref{problem}. 
TNCP is an ADMM algorithm by Liu et al.~\cite{liu2014trace} for solving a matrix trace-norm regularized problem, 
\rev{which is transformed into the form of \eqref{prog:lrtc-b} where $\psi(U)= \sum^{3}_{i=1}\alpha_{i}\|U^{(i)}\|_{*}$.} 
AirCP is an ADMM algorithm by Ge et al.~\cite{ge2016uncovering} for \rev{solving 
\eqref{prog:lrtc-b} 
where $\psi$ takes the form of $\psi(U,X)$---under equality constraints $U^{(i)}=X^{(i)}$---and is composed of the sum of Frobenius norms of $U^{(i)}$ and the graph Laplacian-based norms of $X^{(i)}$.} 
\rev{FIST is a nonnegative matrix factorization (NMF) method by Li et al.~\cite{li2021imputation} for solving \eqref{prog:lrtc-a} with additional nonnegative constraints on all three CP factors $U^{(i)}$, where $\psi$ is the graph Laplacian-based function \eqref{eq:flf} of the vectorization of the (third-order) candidate tensor.} 
\section{Algorithms}\label{sec:algorithm} 

In this section, we introduce an alternating minimization (AltMin) algorithm and
an ADMM algorithm for solving the LRTC problem~\eqref{problem}.

\subsection{Alternating minimization}\label{ssec:alg-altmin-cg}

To minimize the objective function of~\eqref{problem}, defined on the product space of 
\rev{CP factors} 
$\reals^{m_1\times R}\times\cdots\times\reals^{m_k\times R}$, alternating minimization (also referred to as block coordinate descent)  consists in minimizing the function cyclically over each factor matrix among $(U^{(1)},\ldots,U^{(k)})$  while keeping the remaining variables fixed at their last updated values. 

\begin{algorithm}[htbp] 
    \caption{Alternating minimization (AltMin) for solving~\eqref{problem}}\label{algo:altmin}
    \hspace*{0.02in} {\bf Input:} Data (\revn{revealed} on $\Omega$) $\mathcal{P}_\Omega(\mathcal{T})\in\mathbb{R}^{m_{1}\times\ldots \times m_{k}
        }$, observed set $\Omega$. Objective function $f$\\
        \hspace*{0.02in} {\bf Output:} $(U_{t}^{(i)})_{i=1,\ldots,k}$
    \begin{algorithmic}[1]
        \State{Initialization: $U_{0}^{(1)},\ldots,U_{0}^{(k)}$} 
        \For{$t=0,1,2,\ldots$}\label{algo:altmin-init}
        \If{stopping criterion is satisfied}\label{line:altmin-stoppcrit}
        \State return;
        \EndIf{}
        \For{$i=1,\ldots,k$}
        \State{ $U_{t+1}^{(i)} = \arg\min_{U^{(i)}\in\mathbb{R}^{m_i\times R}} f_{t+1}^{(i)}(U^{(i)})$
        \hfill \devv{See Algorithm~\ref{alg:Altmin_CG} or Algorithm~\ref{algorithm_ADMM}}}
        \EndFor
        \EndFor
    \end{algorithmic}
\end{algorithm}

Let $f(\cdot)$ denote the objective function of~\eqref{problem}. \revs{We consider the optimization problem on the product space, i.e., $\min_{U^{(1)},\ldots,U^{(k)}} f(U^{(1)},\ldots,U^{(k)})$. \revn{Given the $t$-th iterate $(U_t^{(1)},\ldots,U_t^{(k)})$, the (cyclic)} alternating minimization resorts to solving the following sequence of subproblems for $i=1,\dots,k$,}
\begin{equation}\label{subequation}
\revs{\min_{U^{(i)}\in\reals^{m_i\times R}}f_{t+1}^{(i)}(U^{(i)}) := f(U_{t+1}^{(1)},\ldots,U_{t+1}^{(i-1)},U^{(i)},U_{t}^{(i+1)},\ldots,U_{t}^{(k)}),}
\end{equation}
where $f_{t+1}^{(i)}$ denotes the objective function of the subproblem in $U^{(i)}$. \revs{The procedure of alternating minimization is listed in~\Cref{algo:altmin}.}

\paragraph*{\bf Alternating subproblems} Due to the graph Laplacian-based regularization term, the major challenge in
solving~\eqref{problem} by the alternating minimization procedure is the
structure of each subproblem~\eqref{subequation}, which is different from those of an
unregularized tensor decomposition problem. 
\revs{Specifically,} during the $(t+1)$-th iteration and for $i\in\irange[k]$, 
the minimization~\eqref{subequation} has the following explicit expression\footnote{For
convenience, we ignore the subscript $t+1$ or $t$ in the \revf{variables
$U^{(j)}$ for all $j=1,\dots,k$, and omit constant terms in the objective}.}: %
\begin{equation}\label{one_variable2}
\min_{U^{(i)}\in\reals^{m_i\times R}} 
\revs{
\frac{1}{2}\|\mathcal{P}_{\Omega^{(i)}}\big(\mathcal{T}_{(i)}- U^{(i)}(( U^{(j)})^{\odot_{j\neq i}})^{\top}\big)\|_{\mathrm{F}}^{2}
} 
+\frac{\lambda_{i}}{2}\trace\big( (U^{(i)})^{\top} L^{(i)} U^{(i)} \big) 
+\sum_{\substack{j=1\\ j\neq i}}^{k}\frac{\lambda_{j}}{2}\|(U^{(n)})^{\odot_{n\neq j}}\|_{\mathrm{F}}^{2}
\end{equation}
\revs{where the first term is transformed from the mode-$i$ matricization~\eqref{eq:unfold}, and $\Omega^{(i)}$ is the set of $2$-dimensional indices of the form  
$(\ell_{i},r_{i})$, which is transformed from the tensor index 
$(\ell_{1},\ldots,\ell_{k})\in\Omega$ via~\eqref{eq:ind-mati} after the mode-$i$ matricization %
(\Cref{sec:preliminaries}).
In fact, the subproblem~\eqref{subequation} has a quadratic objective.
} 

\begin{proposition}
    \label{prop:altmin}
    \rev{Let $\mathbf{x}:=\text{vec}((U^{(i)})^{\top})\in\reals^{m_iR}$ be the vectorization of $(U^{(i)})^{\top}$, and let $g^{(i)}(\mathbf{x}) := f_{t+1}^{(i)} (U^{(i)})$ defined in \eqref{subequation}.} 
    Then $g^{(i)}$ is a quadratic function of the following form, 
\begin{align}
g^{(i)}(\mathbf{x}) & :=\frac{1}{2}\mathbf{x}^{\top}M^{(i)}\mathbf{x} - \mathrm{vec}(Q^{(i)})^{\top}\mathbf{x} %
\quad \text{\revs{with}} \label{g_function}\\ 
M^{(i)}&:= A^{(i)} + I_{m_{i}}\otimes C^{(i)} + \lambda_{i}L^{(i)}\otimes I_{R}  %
\label{M}\\
Q^{(i)} &:= (\mathcal{P}_{\Omega^{(i)}}(\mathcal{T}_{(i)}) )(U^{(j)})^{\odot_{j\neq i}}, %
\label{Q} 
\end{align}
\rev{where $A^{(i)}= \sum_{s=1}^{m_i}(\mathbf{e}_s \mathbf{e}_s^{\top})\otimes A^{(i)}_s \in\mathbb{R}^{m_iR\times m_iR}$ is such that the $R\times R$ blocks are 
\begin{align} \label{Aj}
    A^{(i)}_{s}=\sum_{\ell\in\Omega^{(i)}_{s}} \revs{ (U^{(-i)}_{\ell,:})^{\top}U^{(-i)}_{\ell,:}  \quad \text{for~} s \in [m_i] \text{~and~} 
\Omega^{(i)}_{s}=\{\ell: (s, \ell)\in\Omega^{(i)}\}, } 
\end{align}
and $C^{(i)}\in\reals^{R\times R}$ writes 
\begin{align}\label{def:mat-ci}
C^{(i)} = \sum_{\substack{j=1\\ j\neq i}}^{k}\lambda_{j}\mathrm{diag}\big( (\| U^{(-i,-j)}_{:,\ell} \|^2)_{\ell=1,...,R} \big)
\quad \text{for} \quad 
U^{(-i,-j)}:= (U^{(n)})^{\odot_{n\neq i, j}},
\end{align} 
where $(U^{(n)})^{\odot_{n\neq i, j}}$ denotes the Khatri-Rao product of $U^{(n)}$'s excluding $U^{(i)}$ and $U^{(j)}$. 
} 
\end{proposition}
\begin{proof}
\rev{Given the vectorization $\mathbf{x}=\text{vec}((U^{(i)})^{\top})$, the term $L^{(i)}\otimes I_R$ in \eqref{M} is obtained due to the relation
$(B^{\top}\otimes A)\text{vec}(X)=\text{vec}(AXB)$.}  
The components $A^{(i)}$ and $C^{(i)}$ in~\eqref{M} are defined and computed as
follows. 
Let $A^{(i)}\in\reals^{m_i R\times m_i R}$ be the matrix of the quadratic form 
$\mathbf{x}^{T}A^{(i)}\mathbf{x} := \|\mathcal{P}_{\Omega}\big(U^{(i)}\big((U^{(j)})^{\odot_{j\neq i}}\big)^{\top}\big)\|_F^2$ 
in $\mathbf{x}$. 
\rev{Notice that 
\begin{align}
    \mathbf{x}^{T}A^{(i)}\mathbf{x} &:= \|\mathcal{P}_{\Omega^{(i)}}\big(U^{(i)}(( U^{(j)})^{\odot_{j\neq i}})^{\top}\big)\|_{\mathrm{F}}^{2} 
     =\big\langle U^{(i)} (U^{(-i)})^{\top}, P_{\Omega^{(i)}}(U^{(i)} (U^{(-i)})^{\top})\big\rangle \nonumber \\
    &~= \trace\Big( U^{(-i)} {U^{(i)}}^{T} P_{\Omega^{(i)}}\big(U^{(i)} (U^{(-i)})^{\top}\big) \Big) 
     = \sum_{s=1}^{m_{i}} U^{(-i)} U_{s,:}^{\top} P_{\Omega_s^{(i)}}\big(U_{s,:} (U^{(-i)})^{\top}\big) \nonumber\\ 
    &~= \sum_{s=1}^{m_{i}} U_{s,:} (U^{(-i)})^{\top}P_{\Omega_s^{(i)}}(U^{(-i)}) U_{s,:}^{\top} 
    = \sum_{s=1}^{m_{i}} U_{s,:} \revs{\left(\sum_{\ell\in\Omega^{(i)}_{s}}(U^{(-i)}_{\ell,:})^{\top}U^{(-i)}_{\ell,:} \right)} U_{s,:}^{\top} \label{eqn:ai-l3b}
\end{align}
where $U_{s,:}$ denotes the $s$-th row of $U^{(i)}$ %
and $U^{(-i)}:= (U^{(j)})^{\odot_{j\neq i}}$ is used for clarity (see \Cref{sec:preliminaries}).} 
Equation~\eqref{eqn:ai-l3b} holds \revs{because} %
the projection $P_{\Omega_s^{(i)}}$ (defined on $\reals^{m_{(-i)}}$ by default) applies to each of the columns of $U^{(-i)}$, which has the effect of projecting any row of $U^{(-i)}$ with index $\ell\notin\Omega_s^{(i)}$ to a row of zeros.
Therefore, $A^{(i)}\in \mathbb{R}^{m_{i}R\times m_{i}R}$ is a block diagonal matrix with $m_{i}$ diagonal blocks and each block has the form
\begin{align*}%
    A^{(i)}_{s}=\sum_{\ell\in\Omega^{(i)}_{s}} \revs{ (U^{(-i)}_{\ell,:})^{\top}U^{(-i)}_{\ell,:}  \quad \text{for~} s \in [m_i], } 
\end{align*}
where $\Omega^{(i)}_{s}=\{\ell: (s, \ell)\in\Omega^{(i)}\}$.

\revs{
Now we verify \eqref{def:mat-ci}. %
The component $I_{m_i}\otimes C^{(i)}$ denotes the matrix related to the %
\revs{third term} in~\eqref{one_variable2}: 
$q(U^{(i)}):=\sum_{j\in[k], j\neq i} \lambda_j \| (U^{(n)})^{\odot_{n\neq j}}\|_F^2$,
which satisfies }
\begin{align*}
q(U^{(i)}) &= \sum_{j\neq i} \lambda_j \sum_{\ell=1}^{R} \| U^{(-i,-j)}_{:,\ell} \otimes U^{(i)}_{:,\ell} \|_2^2
= \sum_{j\neq i} \lambda_j \sum_{\ell=1}^{R}  \underbrace{\| U^{(-i,-j)}_{:,\ell}\|_2^2}_{C_{\ell\ell}^{(i,j)}} \trace\big(U^{(i)}_{:,\ell} (U^{(i)}_{:,\ell})^{\top}\big)\\ %
&= \sum_{j\neq i} \lambda_j \trace\big(U^{(i)} C^{(i,j)} (U^{(i)})^{\top}\big)= \trace\big( U^{(i)} (\sum_{j\neq i} \lambda_j C^{(i,j)}) (U^{(i)})^{\top}\big).\nonumber
\end{align*}
\revs{Finally, since} $\trace(X^{\top} C X) = {\text{vec}(X)}^{\top}(I\otimes C)\text{vec}(X)$, the expression of $C^{(i)}$~\eqref{def:mat-ci} yields the identification 
$q(U^{(i)}) =  \bold{x}^{\top}(I_{m_i}\otimes C^{(i)})\bold{x}$ %
for $\bold{x}=\text{vec}((U^{(i)})^{\top})$. 
\end{proof}

The function $g^{(i)}$ in~\eqref{g_function}, and equivalently $f_{t+1}^{(i)}$ of~\eqref{one_variable2}, is strongly convex (\revs{see the proof of~\Cref{theorem_altmin}}) provided that 
$\lambda_i>0$ \devv{("and $\lambda_L>0$" removed)} for $i=1,\ldots,k$. 
\revs{In fact, solving the minimization~\eqref{one_variable2} 
can be rewritten as} the following quadratic minimization problem %
\begin{equation}\label{prog:sub-lsq}
\min_{\mathbf{x}\in\mathbb{R}^{m_{i}R}} g^{(i)}(\mathbf{x}) =\frac{1}{2}\mathbf{x}^{\top}M^{(i)}\mathbf{x} - \text{vec}(Q^{(i)})^{\top}\mathbf{x}.
\end{equation}
\dele{(more details about why solving the linear system with~\eqref{g_function} is tough, compared to the unregularized tensor decomp)}

\begin{algorithm}[h]
\caption{AltMin-CG for solving (\ref{problem})}\label{alg:Altmin_CG}
\hspace*{0.02in} {\bf Input:} Observed tensor $\mathcal{P}_{\Omega}(\mathcal{T})$, graph Laplacian $\mathcal{L}^{(1)},\ldots, \mathcal{L}^{(k)}$, observed set $\Omega$, parameters  $\lambda_{1},\ldots,\lambda_{k}>0$ and $\lambda_{L}\geq0$ \\
\hspace*{0.02in} {\bf Output:} $(U_{t}^{(i)})_{i=1,\ldots,k}$
\begin{algorithmic}[1]
\State{Initialization: $U_{0}^{(1)},\ldots,U_{0}^{(k)}$}
\For{$t=0,1,2,\ldots$}
\If{stopping criterion is satisfied}
\State return;
\EndIf{}
\For{$i=1,\ldots,k$}
\State{Compute: $C^{(i)},Q^{(i)}$ defined in (\ref{def:mat-ci}), (\ref{Q}) and
$(U^{(j)})^{\odot_{j\neq i}}$\label{algo:altmincg-l1}} 
\State\revn{Get $\mathbf{x}_{t+1}^{(i)}$ by approximately solving $\min_{\mathbf{x}}g^{(i)}(\mathbf{x})$~\eqref{prog:sub-lsq} \hfill \# see Algorithm \ref{alg:CG}\label{algo:altmincg-l2}} 

\State \rev{$U_{t+1}^{(i)}= (\text{unvec}(\mathbf{x}_{t+1}^{(i)}) )^{\top}$}\label{algo:altmincg-l3}
\EndFor
\EndFor
\end{algorithmic}
\end{algorithm}

\revs{We consider linear CG for solving problem~\eqref{prog:sub-lsq}. Algorithm~\ref{alg:Altmin_CG} called AltMin-CG is an adaptation of AltMin (\Cref{algo:altmin}) using linear CG as the subproblem solver. 
In \Cref{alg:Altmin_CG} (line \ref{algo:altmincg-l3}), unvec$(\cdot)$ is the operation of turning the $m_i{R}$-dimensional vector into an ${R}\times m_i$ matrix (in the column-major way).}

\paragraph*{\bf \rev{The linear CG solver}} 
\label{AltMin-CG}
Detailed steps for the linear CG adaptation are given in Algorithm~\ref{alg:CG}. 
\begin{algorithm}[!htbp]
\caption{Linear CG for solving~\eqref{prog:sub-lsq}}
\label{alg:CG}
\hspace*{0.02in} {\bf Input:} $M^{(i)}\in\mathbb{R}^{m_{i}R\times m_{i}R}$, 
$Q^{(i)}\in \reals^{m_i\times R}$, initial point $\mathbf{x}_{0}\in\mathbb{R}^{m_{i}R}$, accuracy parameter $\epsilon$, iteration budget $T_{\max}$\\
\hspace*{0.02in} {\bf Output:} $\mathbf{x}_{t}\in\mathbb{R}^{m_{i}R}$\dele{, $t^{*}$}
\begin{algorithmic}[1]
\State{$\mathbf{r}_{0}=\text{vec}(Q^{(i)})-M^{(i)}\mathbf{x}_{0}$}
\For{$t=0,\ldots,T_{\max}$}
\State{Compute: $\|\mathbf{r}_{t}\|$} 
\If{$\|\mathbf{r}_{t}\|\leq \epsilon\|\mathbf{r}_{0}\|$} 
\State{Break}
\EndIf
\If{$t=0$}
\State{$\mathbf{p}_{1}=\mathbf{r}_{0}$}
\Else
\State{$\mathbf{p}_{t+1}=\mathbf{r}_{t}+\frac{\|\mathbf{r}_{t}\|^{2}}{\|\mathbf{r}_{t-1}\|^{2}}\mathbf{p}_{t}$}
\EndIf
\State{Compute: $\mathbf{v}_{t+1}=M^{(i)}\mathbf{p}_{t+1}$} \hfill \# see
Algorithm~\ref{algorithm_Hessian-vector} \label{alg:cg-l12}
\State{Compute: $\alpha=\frac{\|\mathbf{r}_{t}\|^{2}}{\mathbf{p}_{t+1}^{\top}\mathbf{v}_{t+1}}$}
\State{Compute: $\mathbf{x}_{t+1}=\mathbf{x}_{t}+\alpha \mathbf{p}_{t+1}$,
$\mathbf{r}_{t+1}=\mathbf{r}_{t}-\alpha \mathbf{v}_{t+1}$}
\EndFor
\end{algorithmic}
\end{algorithm}
\rev{Note that the matrix-vector product $M^{(i)}\mathbf{x}$ appearing in the linear CG subroutine (\Cref{alg:CG}, line~\ref{alg:cg-l12}) dominates the computational cost \revn{since $M^{(i)}$ is a matrix} of size $m_i{R}\times m_i{R}$. To overcome this computational bottleneck, we take advantage of the structure of $M^{(i)}=A^{(i)} + I_{m_{i}}\otimes C^{(i)} + \lambda_{i}L^{(i)}\otimes I_{R}$ in~\eqref{M} and propose a more efficient way by the following special Hessian-vector multiplication. 
}
Recall that 
$(B^{\top}\otimes A)\text{vec}(X)=\text{vec}(AXB)$, it follows from $\mathbf{x}=\text{vec}((U^{(i)})^{\top})$ that
\begin{align*}
 (L^{(i)}\otimes I_{R})\mathbf{x} &= \text{vec}((U^{(i)})^{\top}L^{(i)}), 
\\ 
 (I_{m_{i}}\otimes C^{(i)})\mathbf{x} &= \text{vec}(C^{(i)}(U^{(i)})^{\top}).
\end{align*}
Thus the larger Hessian-vector multiplication can be implemented by a series of smaller matrix multiplications as follows
\begin{equation}\label{mx}
M^{(i)}\mathbf{x}= \text{vec}(\lambda_{i}(U^{(i)})^{\top}L^{(i)}+C^{(i)}(U^{(i)})^{\top})+A^{(i)}\mathbf{x}.
\end{equation}

\begin{algorithm}[h]
    \caption{Hessian-vector multiplication $M^{(i)}\mathbf{x}$ in the CG method}
\label{algorithm_Hessian-vector}
\hspace*{0.02in} {\bf Input:} $L^{(i)}\in\mathbb{R}^{m_{i}\times m_{i}}$,
$\Omega^{(i)}_{j}$, $C^{(i)}\in\mathbb{R}^{R\times R}$, $U^{(-i)} :=
(U^{(j)})^{\odot_{j\neq i}}\in\revf{\mathbb{R}^{m_{(-i)}\times R}}$,
$\mathbf{x}:=\text{vec}((U^{(i)})^{\top})\in\mathbb{R}^{m_{i}R}$,
$\lambda_{i}\geq 0$.
\\
\hspace*{0.02in} {\bf Output:} $M^{(i)}\mathbf{x}$
\begin{algorithmic}[1]
\State{$X = \text{unvec}(\mathbf{x})\in\mathbb{R}^{R\times m_{i}}$}
\For{$j=1,\dots, m_i$}
    \State{Compute \revs{$N^{(i)}_{:,j}=\sum_{\ell\in\Omega^{(i)}_{j}}(\revs{U^{(-i)}_{\ell,:}X_{:,j}}) (U^{(-i)}_{\ell,:})^{\top}$}} 
\EndFor
\State{Compute: $M^{(i)}\mathbf{x}=\text{vec}(C^{(i)}X+\lambda_{i}XL^{(i)})+\text{vec}(N^{(i)})$ defined in (\ref{mx})}
\end{algorithmic}
\end{algorithm}

\rev{For the computation of $A^{(i)}\mathbf{x}$ in~\eqref{mx}, we make use of
the block diagonal structure of $A^{(i)}$ specified in~\eqref{Aj}. More precisely, let 
$$N^{(i)}_{:,j}:=A_{j}^{(i)}(U^{(i)}_{j,:})^{\top}=\sum_{\ell\in\Omega^{(i)}_{j}}(U^{(-i)}_{\ell,:}(U^{(i)}_{j,:})^{\top}) (U^{(-i)}_{\ell,:})^{\top} \quad\text{for}\quad j\in[m_i].$$ 
Then we have
\begin{equation}\label{hessian_vector}
    A^{(i)}\mathbf{x} = A^{(i)}\text{vec}((U^{(i)})^{\top}) =\text{vec}(N^{(i)})=\text{vec}((N^{(i)}_{:,1},\dots,N^{(i)}_{:,m_i})), 
\end{equation}
where the computation of each $N^{(i)}_{:,j}$ can be done in parallel. 
Details to compute the Hessian-vector product in the CG method are listed in
Algorithm~\ref{algorithm_Hessian-vector}.
}

\paragraph*{\bf Computational cost of AltMin-CG} %

The computational cost for each alternating step \eqref{one_variable2} corresponds to the procedure required by 
line~\ref{algo:altmincg-l1}--line~\ref{algo:altmincg-l3} of Algorithm~\ref{alg:Altmin_CG}. 

The cost of forming $(U^{(j)})^{\odot_{j\neq i}}$ is 
$O(\frac{|\Omega|R}{\rho m_i})$, where $\rho$ denotes the sampling rate.  
The cost of computing $Q^{(i)}$ in (\ref{Q}) is $O(|\Omega|R)$ with access to $(U^{(j)})^{\odot_{j\neq i}}$. 
The cost of forming $C^{(i)}$ in (\ref{def:mat-ci}) is \dele{(to revise)}\dele{(gy) ? this we have discussed, I think it is right and the form is ok}
$O(\frac{|\Omega|R}{\rho m_i m_j})$.

\newcommand{\ncg}{n_{\text{CG}}}
\newcommand{\ncgb}{\tilde{n}_{\text{CG}}}
The major cost in Algorithm~\ref{alg:Altmin_CG} corresponds to line~\ref{algo:altmincg-l2}, which involves (inner) iterations of the linear CG. The total cost of line~\ref{algo:altmincg-l2} is $\ncg$ times the per-iteration cost of the linear CG algorithm (Algorithm \ref{alg:CG}), where $\ncg$ denotes the number of iterations required by the CG solver (Algorithm~\ref{alg:CG}) for producing $\mathbf{x}^{(i)}_{t+1}$. The per-iteration cost of Algorithm~\ref{alg:CG} is mainly composed of the following components. 
\begin{itemize}
\item Cost of computing $A^{(i)}\mathbf{x}$: $O(|\Omega|R)$, since the cost of computing $A_{j}^{(i)}(u^{(i)}_{j,:})^{\top}$ in (\ref{hessian_vector}) is $O(|\Omega^{(i)}_{j}|R)$ for $j=1,...,m_i$ and $\sum_{j=1,...,m_i}|\Omega_{j}^{(i)}| = |\Omega|$; 
\item Cost of computing $(L^{(i)}\otimes I_R)\mathbf{x}$: $O(\text{nnz}(L^{(i)})R)$; 
\item Cost of computing $(I_{m_i}\otimes C^{(i)})\mathbf{x}$: $O(m_{i}R)$. 
\end{itemize}
\revs{Hence, the cost of computing the Hessian-vector multiplication $M^{(i)}\mathbf{x}$ is 
$O(\text{nnz}(L^{(i)})R+|\Omega|R)$. The number of linear CG iterations needed
is theoretically bounded by the problem dimension; and \revn{in practice, we
limit this number by a constant iteration budget.} Therefore the dominant per-iteration cost of \Cref{alg:Altmin_CG} is 
$$O((\text{nnz}(L^{(i)})+|\Omega|)R).$$
}
\begin{remark}
The main computational challenge in finding the
solution of~\eqref{prog:sub-lsq} is the presence of graph Laplacian-based
regularization terms \revs{in~$M^{(i)}\in\mathbb{R}^{m_{i}R\times m_{i}R}$}. The similar difficulty can be found
in the graph-regularized least squares problem in~\cite{rao2015collaborative}.
More precisely, \revs{the matrix $M^{(i)}=A^{(i)} + I_{m_{i}}\otimes C^{(i)} + \lambda_{i}L^{(i)}\otimes I_{R}$} is not block diagonal, 
\rev{due to the fact that the component $L^{(i)}\otimes I_{R}$ is not block diagonal, since $L^{(i)}$ defined in \eqref{eq:def-Li} has nonzeros---corresponding to the edges of the graph $\grp^{(i)}$---on its off-diagonal terms.
Therefore, the minimization problem~\eqref{prog:sub-lsq} cannot be decomposed into $m_i$ separable smaller
problems in $\reals^{R}$.} \label{rmk:3.1}
\end{remark}
In light of Remark~\ref{rmk:3.1}, we also
consider %
an alternating direction method of
multipliers (ADMM) as an alternative way to address the difficulty with this
nonseparable quadratic problem.

\subsection{ADMM}\label{ssec:admm}
\rev{Besides AltMin-CG, we adapt an ADMM algorithm~\cite{ge2016uncovering} to solve the graph-regularized problem~\eqref{problem} \revs{whose objective consists of three terms: the data fitting term, the graph regularizers, and the Frobenius norm-based regularizers}.} 
%
\revn{The advantage of ADMM lies in that it decomposes complex optimization problems into sequences of simpler subproblems and that it splits coupling constraints by a dual multiplier. 
(Another reason for its popularity is that, if the objective function is strongly convex---which is not the case of~\eqref{problem}---and Lipschitz continuous, then with appropriate choice of parameters ADMM will convergence linearly.) 
For nonconvex problems, ADMM can be considered as a local minimization method, and the hope is that it will possibly have better convergence properties
than other local optimization methods~\cite{boyd2011distributed}.}

\revs{Since $L^{(i)}$ brings about the coupling effect to the quadratic minimization problem underlying \eqref{problem},} we introduce $B^{(i)}$ as an auxiliary variable that equals $U^{(i)}$ to decouple the regularization terms in problem (\ref{problem}) as follows
\begin{align}\label{problem_admm}
\min_{U^{(1)},\ldots,U^{(k)}}
&\frac{1}{2}\|\mathcal{P}_{\Omega}(\mathcal{T}-\llbracket U^{(1)},\cdots, U^{(k)}\rrbracket)\|_{\mathrm{F}}^{2}
+\sum_{i=1}^{k}\frac{\lambda_{i}}{2}\trace\big((B^{(i)})^{\top} L^{(i)} B^{(i)}\big) 
+ \sum_{i=1}^{k}\frac{\lambda_{i}}{2}\|(U^{(j)})^{\odot_{j\neq i}}\|_{\mathrm{F}}^{2},\nonumber\\
&\text{subject to} \quad U^{(i)}=B^{(i)},~\text{for}~i=1,\ldots,k.
\end{align}
Then the augmented Lagrangian for the above optimization problem
(\ref{problem_admm}) is
\begin{eqnarray}
&&\mathcal{L}_{\eta}(U,B,Y)=f(U,B)+\sum_{i=1}^{k}\langle Y^{(i)},B^{(i)}-U^{(i)} \rangle+\sum_{i=1}^{k}\frac{\eta}{2}\|B^{(i)}-U^{(i)}\|_{\mathrm{F}}^{2},
\label{lagrange}
\end{eqnarray}
where $f(U,B)$ denotes the objective function of~\eqref{problem_admm}, $Y^{(i)}\in\mathbb{R}^{m_{i}\times R}$ are the Lagrange multipliers, and $\eta > 0$ is a penalty parameter. \revn{Given the current iterates $U_t$ and $B_t$, applying} the ADMM iterative scheme successively to minimize 
$\mathcal{L}_{\eta}$ over $\{U^{(1)},\ldots,U^{(k)}\}$ and $\{B^{(1)},\ldots,B^{(k)}\}$ turns out to
\begin{eqnarray}
&&\{U_{t+1}^{(1)},\ldots,U_{t+1}^{(k)}\}=\argmin_{U^{(1)},\ldots,U^{(k)}}\mathcal{L}_{\eta_{t}}(U^{(1)},\ldots,U^{(k)},B_{t}^{(1)},\ldots,B_{t}^{(k)},Y_{t}^{(1)},\ldots,Y_{t}^{(k)}), \label{update_U}\\
&&\{B_{t+1}^{(1)},\ldots,B_{t+1}^{(k)}\}=\argmin_{B^{(1)},\ldots,B^{(k)}}\mathcal{L}_{\eta_{t}}(U_{t+1}^{(1)},\ldots,U_{t+1}^{(k)},B^{(1)},\ldots,B^{(k)},Y_{t}^{(1)},\ldots,Y_{t}^{(k)}), \label{update_B}\\
&& Y_{t+1}^{(i)} = Y_{t}^{(i)}+\eta_{t}(B_{t+1}^{(i)}-U_{t+1}^{(i)}), \quad i = 1,\ldots,k. \label{Y_update}
\end{eqnarray}
\revs{Next we consider solving these subproblems step by step.}
	
\textbf{Updating $\{U_{t+1}^{(1)},\ldots,U_{t+1}^{(k)}\}$:} 
The optimization problem~\eqref{update_U} can be rewritten as follows when updating $\{U_{t+1}^{(1)},\ldots,U_{t+1}^{(k)}\}$
\begin{equation}\label{update_U1}
\min_{U^{(1)},\ldots,U^{(k)}}\frac{1}{2}\
\|\mathcal{P}_{\Omega}(\mathcal{T}-\llbracket U^{(1)},\ldots,U^{(k)}\rrbracket)\|_{\mathrm{F}}^{2}+\sum_{i=1}^{k}\frac{\lambda_{i}}{2}\|(U^{(j)})^{\odot_{j\neq i}}\|_{\mathrm{F}}^{2}+\sum_{i=1}^{k}\frac{\eta_{t}}{2}\|U^{(i)}-B_{t}^{(i)}-(1/\eta_{t})Y_{t}^{(i)}\|_{\mathrm{F}}^{2}.
\end{equation}
We apply the alternating minimization method to update each $U^{(i)}$ for $i=1,\ldots,k$, while fixing the other variables. 
Then problem (\ref{update_U1}) becomes a quadratic optimization problem. For convenience, we ignore the subscript in the fixed $U^{(j)}$ for $j\neq i$, and the resulting subproblem with respect to $U^{(i)}$ is formulated as
\begin{equation}\label{updateU}
    \min_{U^{(i)}}\frac{1}{2}\|\mathcal{P}_{\Omega^{(i)}}(\mathcal{T}_{(i)}-U^{(i)}\revs{((U^{(j)})^{\odot_{j\neq i}})}^{\top})\|_{\mathrm{F}}^{2}+\sum_{\substack{j=1\\ j\neq i}}^{k}\frac{\lambda_{j}}{2}\|(U^{(n)})^{\odot_{n\neq j}}\|_{\mathrm{F}}^{2}+\frac{\eta_{t}}{2}\|U^{(i)}-B_{t}^{(i)}-\frac{Y_{t}^{(i)}}{\eta_{t}}\|_{\mathrm{F}}^{2},
\end{equation}
\revv{which is separable by rows of $U^{(i)}$. Thus, each row of the new iterate $U_{t+1}^{(i)}$ is the solution to the following linear equation in $\mathbb{R}^R$,} 
\begin{equation}\label{update_U2}
(A_{j}^{(i)}+\eta_{t}I_{R}+C^{(i)})(U^{(i)}_{j,:})^{\top}= \revs{\big((\mathcal{P}_{\Omega^{(i)}} (\mathcal{T}_{(i)}) )(U^{(j)})^{\odot_{j\neq i}}+\eta_{t} B_{t}^{(i)}+Y_{t}^{(i)}\big)}^{\top}_{j,:},
\end{equation}
\revv{where $A_j^{(i)}$ and $C^{(i)}$ are defined in~\eqref{Aj} and~\eqref{def:mat-ci} respectively, for $j=1,\ldots,m_{i}$.} %

\textbf{Updating $\{B_{t+1}^{(1)},\ldots,B_{t+1}^{(k)}\}$:} By alternating minimization method, the optimization
problem (\ref{update_B}) can be reformulated as follows when  updating the variables
$\{B_{t+1}^{(1)},\ldots,B_{t+1}^{(k)}\}$, 
\begin{equation}\label{update_B1}
\min_{B^{(i)}}
{\lambda_{i}}\trace ( (B^{(i)})^{\top} L^{(i)} B^{(i)}) 
+{\eta_{t}}\|U_{t+1}^{(i)}-B^{(i)}-(1/\eta_{t})Y_{t}^{(i)}\|_{\mathrm{F}}^{2},
\end{equation}
\revs{which boils down to solving the linear equation,}
\begin{equation}\label{update_B2}
(\eta_{t}I_{m_{i}}+\lambda_{i}L^{(i)})B^{(i)}=\eta_{t}U^{(i)}_{t+1}-Y_{t}^{(i)}.
\end{equation}
\revn{To update $B_{t+1}^{(i)}$ in~\eqref{update_B2}, we adopt the CG method (\Cref{alg:CG}) combined with the Hessian-vector product presented in \Cref{AltMin-CG}.} 
\revn{Similarly, the iterate $U_{t+1}^{(i)}$ in~\eqref{update_U2} is updated by rows using the same  CG method.} 

In summary, the above procedures are presented in \Cref{algorithm_ADMM}.

\begin{algorithm}
\caption{ADMM for solving (\ref{problem})}
\label{algorithm_ADMM}
\hspace*{0.02in} {\bf Input:} Observed tensor $\mathcal{P}_{\Omega}(\mathcal{T})$, graph Laplacian $\mathcal{L}^{(1)},\ldots, \mathcal{L}^{(k)}$, observed set $\Omega$, parameters \revn{$\gamma$, $\eta_{\max}$}, $\lambda_{1},\ldots,\lambda_{k}$ and $\lambda_{L}$ \\
\hspace*{0.02in} {\bf Output:} $(U_{t}^{(i)})_{i=1,\ldots,k}$
\begin{algorithmic}[1]
\State{Initialization: $U_{0}^{(1)},\ldots,U_{0}^{(k)},\eta_{0}$}
\For{$t=0,1,2,\ldots,$}
\If{stopping criterion is satisfied}
\State Break 
\EndIf{}
\For{$i=1,\ldots,k$}\label{update_start}
\For{$j=1,\ldots,m_{i}$}
\State {Update the $j$-th row of $U_{t+1}^{(i)}$ by \revs{solving} \eqref{update_U2}}\label{line_update_U}
\EndFor
\State {Update $B_{t+1}^{(i)}$ by \revs{solving} \eqref{update_B2}}\label{line_update_B}
\State {$Y_{t+1}^{(i)}=Y_{t}^{(i)}+\eta^{(i)}_{t}(B^{(i)}_{t+1}-U_{t+1}^{(i)})$}
\EndFor \label{update_end}
\State \revn{Update $\eta_{t+1}= \min(\gamma \eta_{t}, \eta_{\max})$} 
\EndFor
\end{algorithmic}
\end{algorithm}

\paragraph*{\bf Computational cost of ADMM} 
We analyze the computational cost for each alternating step of the augmented
Lagrangian step (\ref{lagrange}) corresponding to the procedure required by
line~\ref{update_start}--line~\ref{update_end} of
Algorithm~\ref{algorithm_ADMM}. The costs of forming $(U^{(j)})^{\odot_{j\neq i}}$, $C^{(i)}$ and $Q^{(i)}$ are computed  
in the complexity analysis part of Section~\ref{AltMin-CG}. 
\rev{The dominant costs for solving the linear equations~\eqref{update_U2} and~\eqref{update_B2}) are $\ncgb (m_{i} + |\Omega|)R$ and
$\ncgb\text{nnz}(L^{(i)})R$ respectively, 
where $\ncgb$ 
denotes the maximal number of iterations required for
solving the linear equations (\ref{update_U2}) and (\ref{update_B2}).} 
\rev{Similar to AltMin-CG, the number $\ncgb$ is theoretically bounded by the
problem dimension, and in practice, \revn{is limited by a constant iteration
budget}, the per-iteration cost of Algorithm~\ref{algorithm_ADMM} (ADMM) is 
$$O((\text{nnz}(L^{(i)}) + |\Omega|)R),$$ 
which is of the same order as Algorithm~\ref{alg:Altmin_CG} (AltMin-CG).}

\section{Convergence analysis}\label{sec:convergence} 
In this section, we will show the global convergence of \revn{\Cref{algo:altmin} (AltMin) to a critical point.} 
\revn{It will come as a consequence of the results in~\cite[section~2]{xu2013block}.} 
\devv{\footnote{\dev{We go with "critical points". The qualification of "local minimizer" or "global
solution" depends on the properties of the sampling set and the ground-truth
tensor of the completion problem; while the KL property used in this section
and the convergence of AltMin to critical points does not depend on these specific properties.}}In contrast, the iterate sequence generated by
Algorithm~\ref{algorithm_ADMM} (ADMM) converges to the KKT point of
(\ref{problem_admm}).}

\subsection{Preliminaries}
\revv{The following definitions and lemmas (\cite{rockafellar2009variational}, \cite[Definition 1]{attouch2010proximal}) 
are used for the convergence analysis in the next subsection.}

\begin{definition} 
Let $f: \mathbb{R}^{m}\mapsto \mathbb{R} \cup\{+\infty\}$ be proper and lower semicontinuous. 
(i) The \emph{domain} of $f$ is defined and denoted by $\text{dom}f:=\{\mathbf{x}\in\mathbb{R}^{m}: f(\mathbf{x})<+\infty\}$.  
(ii) For each $\mathbf{x}\in\text{dom}f$, the \emph{Fr\'echet subdifferential} of $f$ at $\mathbf{x}$, denoted as $\hat{\partial} f(\mathbf{x})$, is defined as follows: 
	 \begin{align*}
	 	\hat{\partial}f(\mathbf{x})=\left\{\mathbf{\xi}\in\mathbb{R}^{m}:
        \liminf_{\substack{\mathbf{y}\neq \mathbf{x}\\ \mathbf{y}\to \mathbf{x}}} \frac{f(\mathbf{y})-f(\mathbf{x})-\langle \mathbf{\xi},\mathbf{x}- \mathbf{y}\rangle}{\|\mathbf{x}- \mathbf{y}\|}\geq 0 \right\}.
	 \end{align*}
If $\mathbf{x}\notin \text{dom}f$, then $\hat{\partial} f(\mathbf{x})=\emptyset$. 
(iii) The \emph{limiting subdifferential} of $f$ at $\mathbf{x}\in \text{dom}f$, denoted as $\partial f(\mathbf{x})$, is defined as follows\devv{~(to verify)} \cite{mordukhovich2006variational}
     \begin{align*}
     	\partial f(\mathbf{x}):=\{\mathbf{\xi}^{*}\in \mathbb{R}^{m}:\exists
        (\mathbf{x}_{n})_{n\geq 0}, \mathbf{x}_n\to \mathbf{x},
        f(\mathbf{x}_{n})\to f(\mathbf{x}), ~\text{s.t.}~\exists
        \mathbf{\xi}_{n}\in\hat{\partial}f(\mathbf{x}_{n}), \mathbf{\xi}_n\to \mathbf{\xi}^{*}\}.
     \end{align*}
\end{definition}

\begin{definition}[{K{\L} function~\cite[Definition 2.5]{xu2013block}}]\label{kl}
A function $f(\mathbf{x})$ satisfies the Kurdyka-{\L}ojasiewicz (K{\L}) property at point $\bar{\mathbf{x}}\in\text{dom}(\partial f)$ if, in a certain neighborhood $\mathcal{U}$ of $\bar{\mathbf{x}}$, there exists $\psi(s)=cs^{1-\theta}$ for some $c>0$ and $\theta\in[0,1)$ such that the K{\L} inequality below holds:
\begin{align*}
\psi'(f(\mathbf{x})-f(\mathbf{x^{*}}))\text{dist}(0,\partial f(\mathbf{x}))\geq 1, \text{ for any } \mathbf{x}\in \mathcal{U}\cap \text{dom}(\partial f) \text{ and } f(\mathbf{x})\neq f(\mathbf{x^{*}}),
\end{align*}
where $\text{dom}(\partial f)=\{\mathbf{x}: \partial f(\mathbf{x}) \neq  \emptyset)$ and $\text{dist}(0,\partial f(\mathbf{x}))=\min\{\|\mathbf{y}\|:\mathbf{y}\in\partial f(\mathbf{x})\}$.\\
If $f$ satisfies the K{\L} property at each point of $\text{dom}(f)$,
$f$ is called a K{\L} function.
\end{definition}

\begin{definition}[Strong convexity]\label{sconvexity}
A differentiable function %
$f\revf{:\text{dom}f}\mapsto\mathbb{R}$
is strongly convex if and only if \margn{dom$f$ suffices. dom$f$ first appears in Def.5.1, 1).}
\begin{align*}
f(\mathbf{y})\geq f(\mathbf{x})+ \langle\nabla f(\mathbf{x}), \mathbf{y}-\mathbf{x}\rangle +\frac{\mu}{2}\|\mathbf{y}-\mathbf{x}\|^{2}
\end{align*}
holds for some $\mu>0$ and all $\mathbf{x}, \mathbf{y}\in\revf{\text{dom}f}$.
\end{definition}

\begin{definition}[Coercivity]\label{coercive}
    A \revf{real-valued} function $f:\mathbb{R}^{m}\to \revf{\mathbb{R}}$ %
    is called coercive if and only if 
$f(\mathbf{x})\to +\infty$ as $\|\mathbf{x}\|\to +\infty$.
\end{definition}

\subsection{Convergence properties of AltMin}

\revv{The following two lemmas follow directly from Theorem 2.8 and Theorem 2.9
of~\cite{xu2013block} and are used for proving the convergence of the proposed
alternating minimization method.}

\begin{lemma}\label{first_order}
Assume $f$ satisfies the K{\L} property and $\nabla f$ is Lipschitz continuous on any bounded subset of its domain. Let $(U_{0}^{(1)},\ldots,U_{0}^{(k)})$ be any initialization and $(U_{t}^{(1)},\ldots,U_{t}^{(k)})$ be the sequence generated
by Algorithm~\ref{algo:altmin}, where each subproblem $f_{t}^{(i)}(U^{(i)})$
(line 7 of Algorithm~\ref{algo:altmin}) is
strongly convex \revv{and is solved exactly}. If the sequence $(U_{t}^{(1)},\ldots,U_{t}^{(k)})$ is bounded and there exists a finite limit point $(U_{*}^{(1)},\ldots,U_{*}^{(k)})$, then it converges to $(U_{*}^{(1)},\ldots,U_{*}^{(k)})$, which is a critical point of $f$.
\end{lemma}

The convergence rate of the sequence is as follows.
\begin{lemma}\label{lemma:convergence_rate}
Assume $\nabla f$ is Lipschitz continuous on any bounded set and suppose that $U_{t}^{(i)}$ converges to a critical point $U_{*}^{(i)}$ for $i=1,\ldots,k$, at which $f$ satisfies the K{\L} inequality with $\psi(s) = cs^{1-\theta}$ for $c > 0$ and $\theta \in  [0, 1)$.
We have:\\
1. If $\theta = 0$, $U_{t}^{(i)}$ converges to $U_{*}^{(i)}$ in a finite number of iterations;\\
2. If $\theta \in  (0, \frac{1}{2}]$, $\|U_{t}^{(i)}-U_{*}^{(i)}\|\leq \beta\tau^{t}$, $\forall t \geq t_{0}$ for certain $t_{0} > 0$, $\beta > 0$, $\tau\in [0, 1)$;\\
3. If $\theta \in  (\frac{1}{2},1)$, $\|U_{t}^{(i)}-U_{*}^{(i)}\|\leq \beta t^{-(1-\theta)/(2\theta-1)}$, $\forall t \geq t_{0}$ for certain $t_{0} > 0$, $\beta > 0$.\\
Part 1, 2 and 3 correspond to finite convergence, linear convergence, and sublinear convergence, respectively.
\end{lemma}

\revv{We show that the iterates generated by Algorithm~\ref{algo:altmin}
(AltMin)\devv{\ref{alg:Altmin_CG}} converge to a stationary point in the
following theorem.
\devv{by checking all the assumptions in Lemma \ref{first_order}. While, the
sequence generated by Algorithm~\ref{algorithm_ADMM} globally converges to KKT
point.} 
Note that this theorem applies to Algorithm~\ref{alg:Altmin_CG} (AltMin-CG),
provided that the updated iterate of each of the subproblems (line
\ref{algo:altmincg-l2} in Algorithm~\ref{alg:Altmin_CG}) is the exact minimizer of the corresponding (graph-regularized) least-squares problem.
In practice, this requires setting a sufficiently low tolerance parameter $\epsilon$ for the
subproblem solver (Algorithm~\ref{alg:CG}).}

\begin{theorem}\label{theorem_altmin}
The iterates $(U_{t}^{(1)},\ldots,U_{t}^{(k)})$ generated by
Algorithm~\ref{algo:altmin} (AltMin) from any initialization converge globally
to a critical point of $f$ in~\eqref{problem}. Moreover, linear convergence and
sublinear convergence in parts 2 and 3 of Lemma~\ref{lemma:convergence_rate}
apply depending on $\theta$ in K{\L} property of $f$.
\end{theorem}
\begin{proof}
According to Lemma~\ref{first_order}, we need to check whether all the
assumptions satisfied.

1) Function $f$ in (\ref{problem}) is a K{\L} function with $\theta\in [1/2,1)$
as it is a combination of polynomials which are one kind of real analytic
functions (see \cite[Definition 1.1.5]{krantz2002primer}). The real analytic
function itself and the finite sum or product of real analytic functions are
K{\L} functions, see \cite[section 4]{attouch2010proximal} and \cite[section
2.2]{xu2013block}.

2) Gradient $\nabla f$ is Lipschitz continuous on any bounded subset of domain
since $f$ is a $\mathbf{C}^{\infty}$ function.

3) For $i=1,\ldots,k$, $f^{(i)}$ in~\eqref{one_variable2} is strongly convex by Definition~\ref{sconvexity} since $L^{(i)}$ in~\eqref{eq:def-Li} is positive definite. Therefore, the quadratic form $g^{(i)}$ of the subproblems~\eqref{g_function} is strongly convex through the identification
$g^{(i)}(\text{vec}((U^{(i)})^{\top})) = f^{(i)}(U^{(i)})$. Moreover,
the solution for each $g^{(i)}$ corresponds to the exact minimizer.

4) Notice that since $f$ is coercive as defined in Definition~\ref{coercive} and
real analytic, it is guaranteed to produce a bounded sequence
$(U_{t}^{(1)},\ldots,U_{t}^{(k)})$, thus it has a critical point
$(U_{*}^{(1)},\ldots,U_{*}^{(k)})$.

Lemma~\ref{first_order} then implies that \revf{the sequence 
generated by Algorithm~\ref{alg:Altmin_CG} from any initial point converges %
to} a critical point $(U_{*}^{(1)},\ldots,U_{*}^{(k)})$ of $f$.
Moreover, the asymptotic convergence rates in parts 2 and 3 of
Lemma~\ref{lemma:convergence_rate} apply as $\theta\in [1/2,1)$.
\end{proof}

\section{Experiments}\label{sec:numerical} 
In this section, we carry out some numerical experiments to demonstrate the
\revn{workings} of our proposed algorithms Algorithm \ref{alg:Altmin_CG} (AltMin-CG) and Algorithm \ref{algorithm_ADMM} (ADMM) on the LRTC model (\ref{problem}). 
All numerical experiments were performed on a Macbook Pro with a 2.3 GHz Intel Core i7 CPU, 16GB RAM and MATLAB R2015a with Tensor Toolbox version 2.5~\cite{bader2012matlab}. The source code
is made available online.\footnote{
\url{https://gitlab.com/ricky7guanyu/tensor-completion-with-regularization-term}.} %

\rev{First, we test the effect of the graph Laplacian regularization in the LRTC model~\eqref{problem} and we compare the recovery quality of solutions to the graph-regularized tensor completion model to two other models without graph regularization. 
Then, we evaluate time efficiency of the proposed methods (\Cref{alg:Altmin_CG} and \Cref{algorithm_ADMM}) in optimizing the graph-regularized model \eqref{problem}, and compare them with the baseline methods on both synthetic data and real data.
}

\newcommand{\ova}{\Omega_{\text{val}}}
\newcommand{\otr}{\Omega_{\text{tr}}}

\subsection{\rev{Experimental methodology and datasets}}\label{ssec:exp-data-gr}
In the experiments, we evaluate the recovery performance of a tensor completion
result \revn{$\hat{\mathcal{T}}:=\llbracket U^{(1)},\cdots,U^{(k)}\rrbracket$
\revn{with the following error functions restricted on an index set 
$\Omega'$---which contains the revealed entries or the unrevealed entries for test error---of $\mathcal{T}$:} 
(i) the relative error of $\hat{\mathcal{T}}$ 
against $\mathcal{T}$ in the Frobenius norm, and (ii) the root mean squared error (RMSE) of $\hat{\mathcal{T}}$. 
The training and test RMSEs refer to the RMSE on the training set ($\otr$)---the set of revealed entries for the optimization of the model---and the test set respectively.}

We initialize both our proposed methods and other methods with a point $U_0\in\reals^{m_1\times R}\times\reals^{m_2\times R}\times\reals^{m_3\times R}$ where each factor matrix $U_{0}^{(i)}$ is a Gaussian matrix such that $[U_{0}^{(i)}]_{jr}\sim \mathcal{N}(0,1)$.

\rev{The stopping criterion in Algorithm~\ref{algo:altmin} (line~\ref{line:altmin-stoppcrit}) is satisfied if either of the following conditions is met: %
(i) the wall time used for producing the latest iterate is larger than a time budget parameter $T_{\max}$; %
(ii) the progress of the iterate $(U_{t}^{(i)})_{i=1,\ldots,k}$, measured by a heuristic difference function $\Delta_{t}$,  
is smaller than a tolerance parameter $\epsilon$\devv{~(to update its notation if necessary)}. Here we define $\Delta_{t}$ as follows,
\begin{align}\label{eq:stopp-crit-delta}
\Delta_{t} := |E(U_{t};\otr) - E(U_{t-1};\otr)| \quad \text{with} \quad 
E(U;\otr):= \frac{\|P_{\otr}(\llbracket U^{(1)},...,U^{(k)}\rrbracket - \mathcal{T})\|_{\mathrm{F}}}{\|P_{\otr}(\mathcal{T})\|_{\mathrm{F}}}, 
\end{align}
where $E(U;\otr)$ is the relative error restricted on the training set %
$\otr$. %
}

Based on a ground truth tensor 
$\mathcal{T}\in\mathbb{R}^{m_1\times m_2\times m_3}$
and for a fixed sampling rate, we generate $N_{\text{test}}$ training instances $(\Omega_\ell)_{\ell=1,...,N_{\text{test}}}$ under the same sampling rate. For each training instance $(\mathcal{T},\Omega_\ell)$, $N_{\text{init}}$ initial points $(U_{0, (\ell,j)})$, for $j=1,..,N_{\text{init}}$, are generated. Let $\hat{\mathcal{T}}(U_{0,(\ell,j)}; \Omega_\ell)$ denote the solution of the $j$-th test based on the training instance $(\mathcal{T},\Omega_\ell)$ with the initial point $U_{0,(\ell,j)}$, and $E(\hat{\mathcal{T}})$ the error (e.g., relative error, RMSE) of the candidate tensor $\hat{\mathcal{T}}$ w.r.t. $\mathcal{T}$. Then each method is evaluated by the following score
\begin{align*}%
\bar{E}(\hat{\mathcal{T}}) = \frac{1}{N_{\text{test}}N_{\text{init}}}\sum_{\ell=1}^{N_{\text{test}}} \sum_{j=1}^{N_{\text{init}}} E(\hat{\mathcal{T}}(U_{0,(\ell,j)}); \Omega_\ell).
\end{align*}
In all experiments we set $N_{\text{test}}=5$ and $N_{\text{init}}=10$.

In all experiments, the problem-related parameters $(\lambda_i, \lambda_{L})$ in~\eqref{problem}--\eqref{eq:def-Li} 
are generated randomly with the uniform distribution in the log scale. Then the
parameter is chosen among all generated parameter settings through $K$-fold
cross validation (for $K=3$).

\paragraph*{\bf Synthetic data}
\rev{To investigate the effects of graph regularization in the tensor completion model (\ref{problem}), it is tempting to generate a synthetic low rank tensor $\mathcal{T}$ that can be related to structural information of a certain graph.} 
For this purpose, we consider the following tensor model. First, \rev{we generate $k$ Gaussian matrices $(U^{(1)},\dots,U^{(k)})$ of size $m_i\times R$, for a rank parameter $0< R < \min(m_i)$.} Then, we generate a graph Laplacian matrix $\mathcal{L}^{(1)}$ \rev{on the row index set of $U^{(1)}$ %
from the ``Community'' graph model (a type of graphs that contain a number of different closely connected subgraphs or clusters) of GSPbox~\cite{perraudin2014gspbox}}. 
\rev{The structural information of the graph Laplacian is obtained through the eigenvalue value decomposition $\mathcal{L}^{(i)}:= \mathcal{U}\Lambda\mathcal{U}^{\top}$, and is affected to the final tensor model as follows 
\begin{align}
    \label{eq:synth-tensor}
\mathcal{T}:=\llbracket \tilde{U}^{(1)}, U^{(2)}, \dots, U^{(k)}\rrbracket + \mathcal{E} 
\quad \text{for}\quad \tilde{U}^{(1)} := \mathcal{U}\Lambda^{-1}U^{(1)}, 
\end{align}
where $\mathcal{E}$ represents a tensor containing 
additive noise such that
$\mathcal{E}_{\ell_{1}\ell_{2}\ell_{3}}\sim\mathcal{N}(0,\sigma)$,} where
$\sigma$ is set by a given signal-to-noise ratio (SNR). In the experiments, the SNR is set as $20$ dB. 

\rev{The tensor model~\eqref{eq:synth-tensor}, in the same spirit as the synthetic matrix model in~\cite{rao2015collaborative,dong2021riemannian}, has the effect of creating pairwise similar entries (along the first dimension) according to the graph connections. 
In fact, the columns of $\tilde{U}^{(1)}=\mathcal{U}\Lambda^{-1}U^{(1)}$, by construction, \revn{belong} to the eigensubspace of the given graph Laplacian matrix $\mathcal{L}^{(1)}$ with a prescribed order of %
weights in the direction of each eigenvector; here the weights are given by $(\Lambda_{ii}^{-1})_{i=1,\ldots,m_1}$ which make the eigenvectors with the lowest eigenvalues the most significant components of $\tilde{U}^{(1)}$. Therefore, based on graph spectral analysis\comm{~[ref]}, each column of $\tilde{U}^{(1)}$ is a smooth-varying function on the given graph.}
 
\rev{In the experiments third-order tensors are tested with dimensions set as $(m_1,\dots,m_3)=(100,100,100)$, and the rank parameter in \eqref{eq:synth-tensor} is set as $R=10$. 
}

\paragraph*{\bf MovieLens dataset}
The MovieLens-100k dataset\footnote{https://grouplens.org/datasets/movielens/100k/} consists of $100,000$ movie ratings from $943$ users on $1682$ movies during a seven-month period from September 19th, 1997 through April 22nd, 1998. Each movie rating in this dataset has a time stamp. 
Therefore, we obtain a tensor $\mathcal{T}$ of size $943 \times 1682 \times 7$
({\it i.e.}, time period is split into $7$ \revf{parts}).
We randomly select $80\%$ of the known ratings as training set. %

\paragraph*{\bf FIA dataset}
In this experiment, we use the ``rank-deficient spectral FIA
dataset''.\footnote{http://www.models.life.ku.dk/datasets} This dataset consists of results of flow injection analysis on $12$ different chemical substances. The represented tensor is of size $12\text{~(substances)} \times 100 \text{~(wavelengths)} \times 89 \text{~(reaction times)}$.

\subsection{Graph-regularized tensor completion}\label{ssec:exp-recerr-models}

\rev{In this subsection, we evaluate the graph-regularized tensor completion model in comparison to models without graph Laplacian regularization. 
Hereafter, we label the graph Laplacian regularizer of the tensor completion model~\eqref{problem} as %
\greg, which also denotes by abuse the model~\eqref{problem} with $\lambda_{i}\neq 0$ and $\lambda_{L}\neq 0$.} 
The Frobenius norm-based regularizers in~\eqref{problem} is labeled as 
\rev{\freg, which denotes the tensor completion model~\eqref{problem}
when the graph Laplacian regularizer is reduced to zero, i.e., $\lambda_{L}=0$ while the other regularizers are active ($\lambda_{i}\neq 0$).} 
\rev{The unregularized model, i.e.,~\eqref{problem} when all regularization parameters $\lambda_{i}=\lambda_{L}=0$, 
is labeld as %
\noreg.} 
\devv{Besides our proposed AltMin-CG (Algorithm~\ref{alg:Altmin_CG}) and AltMin-ADMM (Algorithm~\ref{algorithm_ADMM}), we also test another alternating minimization method called "AltMin" whose outer loop uses the same procedure as Algorithm~\ref{algo:altmin} but directly solves each subproblem (\ref{prog:sub-lsq}) using backslash in matlab code (i.e., uses an exact linear solver).} 
In the experiments of this subsection, the stopping criterion is controlled by (i) a large enough iteration budget $T_{\max}$; and (ii) a global tolerance parameter $\epsilon$ for~\eqref{eq:stopp-crit-delta}. 

\paragraph*{\bf Experiment on synthetic data} 

\rev{In this experiment, we consider tensors generated from the synthetic model~\eqref{eq:synth-tensor}, which incorporates the graph Laplacian information of a given graph. We conduct tensor completion tests for different sampling rates in $\{0.3\%$, $0.5\%$, $0.7\%, 1\%\}$. 
The rank parameter for all LRTC models is set as $R=10$. 
} 
The average relative errors %
are listed in Table~\ref{table_syn2}. Figure~\ref{fig_syn2_a} shows the histogram with $\mathrm{SR}=0.3\%$. 
We see that in this experiment where AltMin-CG and ADMM are used, the \greg\ model outperforms the other two models.
 
\begin{table}[htbp]
\centering
\small
\caption{Average recovery accuracy (relative error) of the three tensor completion models on synthetic data: \greg\ (with graph Laplacian), \freg\ (without graph Laplacian), \noreg\ (no regularizer).}
\label{table_syn2}
\begin{tabular}{ccccc}
 \toprule
 Sampling Rate       &     \multicolumn{1}{l|}{Algorithm}          & \greg\ & \freg\ & \noreg\ \\ \midrule
 \multirow{2}*{$0.3\%$}                & AltMin-CG & 0.8198  & 1.0083  &  7.1110        \\ \cmidrule(lr){2-5} 
       & ADMM &  0.8246  &  1.0054 &  4.4393        \\ \midrule
   \multirow{2}*{$0.5\%$}              & AltMin-CG &  0.4418  &   0.9201  &   9.9010       \\ \cmidrule(lr){2-5} 
     & ADMM &  0.4486     & 0.9236   &   7.6692          \\  \midrule
  \multirow{2}*{$0.7\%$}               & AltMin-CG &  0.3106  & 0.7561  & 9.9548       \\ \cmidrule(lr){2-5} 
    & ADMM &   0.3000    & 0.7487  & 8.4821    
\\ \midrule
  \multirow{2}*{$1\%$}              & AltMin-CG & 0.1380     & 0.4772  &   13.5740        \\ \cmidrule(lr){2-5} 
      & ADMM & 0.1439    &  0.4513  &  10.9230
\\  
\bottomrule
\end{tabular}
\end{table}


\begin{figure}[h]
\centering
\includegraphics[width=0.37\textwidth]{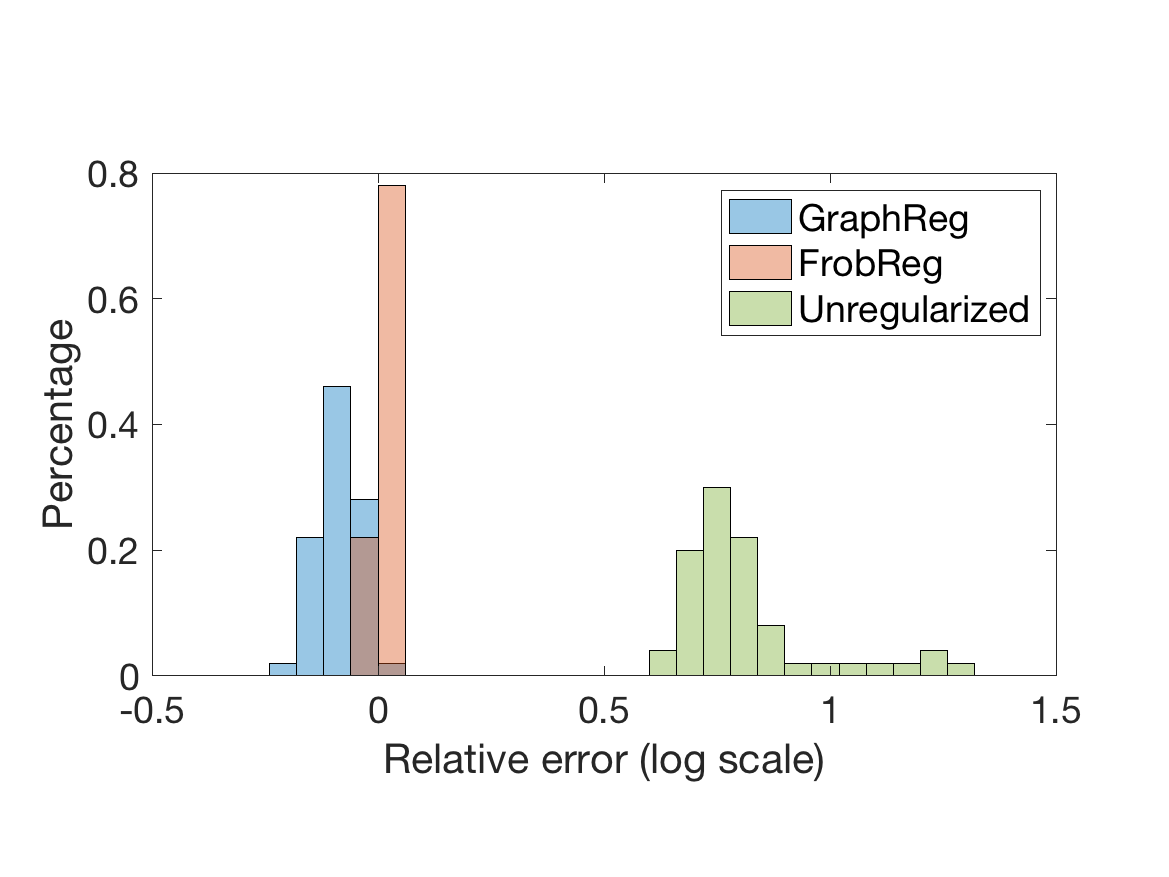}
\quad\qquad
\includegraphics[width=0.37\textwidth]{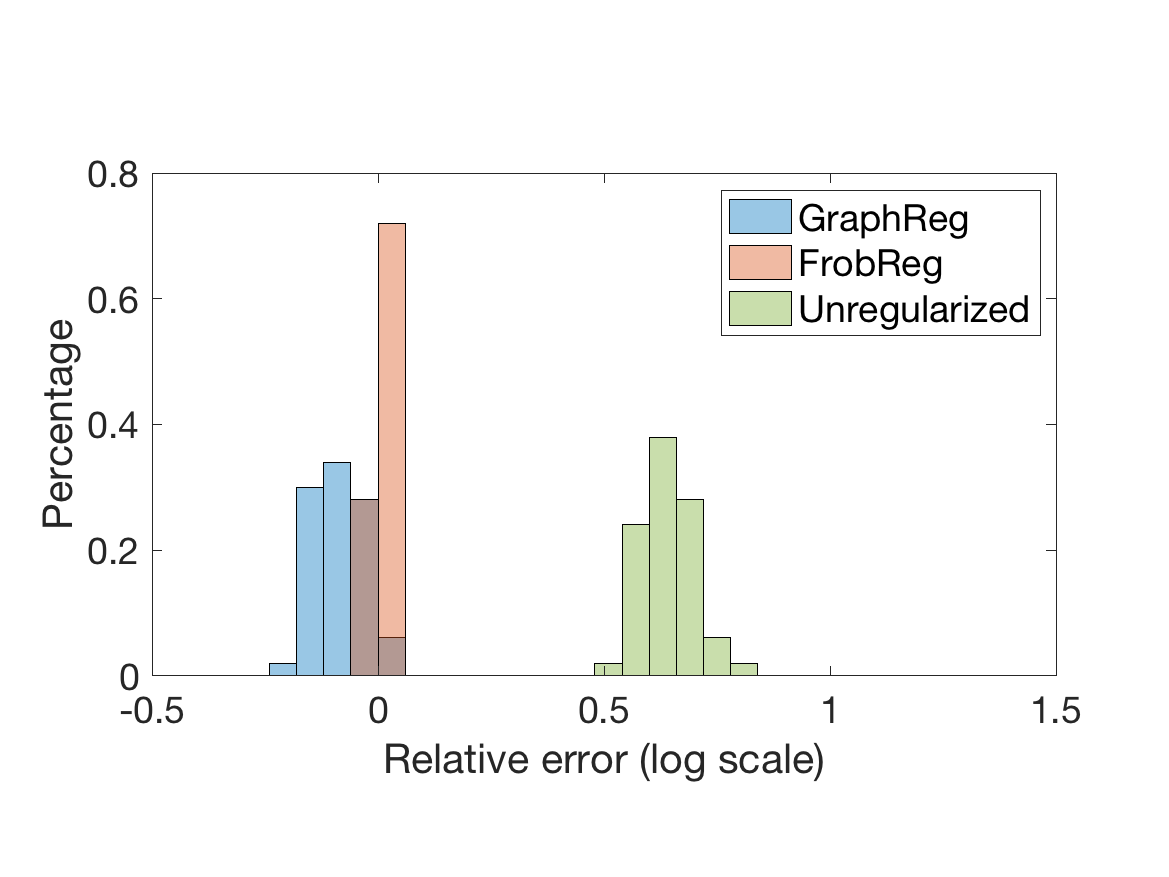}
\caption{Histogram of recovery errors of the three tensor completion models on synthetic data. The sampling rate is $0.3\%$. 
(Left) AltMin-CG; (Right) ADMM. 
}
\label{fig_syn2_a}
\end{figure}

\paragraph*{\bf\revs{Experiment on the MovieLens dataset}} 
Next, we conduct the same experiment on the MovieLens dataset. 
\rev{An essential difference with the previous experiment on synthetic data is that the desired similarity graph for regularizing the tensor completion model is not directly available but requires either a collection of auxiliary graph information or construction from data with a certain graph model}. 

\rev{In the case of the MovieLens dataset, we assume the similarity patterns among the revealed data entries already contains information for constructing a similarity graph on the set of users or the set of items. The reasoning is that for a pair of users that have already given the same (or similar) scores on a same subset of movies, then it is likely that they have similar preferences, and hence they are supposed to be connected for a user-similarity graph.
The similar argument applies to pairs of movies that have very close ratings from users.} 
\rev{Therefore,} we construct a movie-wise similarity graph based on the
data matrix itself (with missing entries). Let $M^{\star}$ denote the MovieLens
data matrix with missing entries. We compute the graph proximity parameters
based on a low-rank approximation of the partially revealed matrix. More
precisely, we use a rank-$r$ approximation of the zero-filled matrix
$M_0:=P_\Omega(\mstar)\in\mathbb{R}^{m\times n}$ as the features for constructing the graph.
Let $(U_0,S_0,V_0)$ denote the $r$-SVD of $M_0$ and let 
$\widetilde{M_0}:= U_0 S_0 V_0^{T}$. 

Next, the computation of the graph edge weight parameters based on the given matrix $M:=\widetilde{M_0}$ 
can be realized using various node proximity methods such as $K$-Nearest Neighbors ($K$-NN) and \revv{$\varepsilon$-graph}\devv{ (removed "NN" according to several references)} models~\cite{chazelle1983improved,belkin2003laplacian,he2004locality,Chen2009}\devv{ (references added; the "Gaussian eps-graph" below is covered by these references)}, which boil down to computing a
certain distance matrix between the rows (resp.\ columns) of $M$. 
Let $Z^{\mathrm{r}}(M)\in\reals^{m\times
m}$ denote the row-wise distance matrix of $M$ defined as 
\rev{$Z_{ij}(M) = \text{dist}( \rowof{M}{i}, \rowof{M}{j} )$, for $i,j\in[m]$, 
where $\text{dist}:\reals^n \times \reals^n\mapsto \reals_+$ is a distance on the
$n$-dimensional vector space}. Subsequently, we build a Gaussian $\varepsilon$-graph by computing the node proximity weights as follows  
\begin{align}
     \label{eq:epsNN-Gaussian-dense}
     (W_\varepsilon(M))_{ij} = \exp\big( -\varepsilon^{-2} Z_{ij}(M) \big) \quad \text{~for~} i,j\in [m]
 \end{align}
where $\varepsilon\in\reals$ is a hyperparameter of the graph model.
Furthermore, a sparse graph adjacency matrix is preferable to a dense one from a computational point of view, as the per-iteration cost of the proposed algorithms (\textit{e.g.}\ Algorithm~\ref{alg:Altmin_CG}) depends
partly on $\text{nnz}(L^{\mathrm{r}})$ and $\text{nnz}(L^{\mathrm{c}})$. For simplicity, we sparsify the graph adjacency matrix
defined in~\eqref{eq:epsNN-Gaussian-dense} with the following thresholding operation 
\begin{align}
     \label{eq:epsNN-Gaussian}
     (W_{\varepsilon,\sigma}(M))_{ij} = \bold{1}_{\geq\sigma}\big( \exp( -\varepsilon^{-2} Z_{ij}(M) )\big) \quad \text{~for~} i,j\in [m], 
 \end{align}
where $\bold{1}_{\geq\sigma}$ is the hard threshold function such that $\bold{1}_{\geq\sigma}(z) = z$ if $z\geq\sigma$ and $0$ otherwise. 
In the graph model~\eqref{eq:epsNN-Gaussian}, parameter $\varepsilon$ is tuned according to the variance of $(Z_{ij})_{i,j=1,..,m}$ and threshold $\sigma$ is chosen according to a preset sparsity level $\mathfrak{s}\ll 1$ for the edge set associated with $W_{\epsilon,\sigma}$ such that $\frac{|\gre(W_{\epsilon,\sigma})|}{m^2}\leq \mathfrak{s}$. 

\revs{The rank parameter for all three LRTC models (\greg, \freg\ and \noreg) is set as $R=10$.} 
The results %
are given in Table \ref{table_real1} and the histogram of these results is presented in Figure \ref{fig_real1}.  
\revv{These results show that the solutions to the graph-regularized model~\eqref{problem} (labeled \greg) and \freg\ model have better recovery performance (in terms of RMSE) than \noreg\ model.}
\revv{The gain of recovery performance induced by the graph learned from the data may
be considered marginal; however, on movie rating data, it is notoriously
difficult to improve\revv{~the RMSE score} much beyond basic methods~\cite{bennett2007netflix}.} %

\begin{table}[htbp]
\centering
\caption{Relative error of the three models on the MovieLens dataset: \greg\ (with graph Laplacian), \freg\ (without graph Laplacian), \noreg\ (no regularizer).}
\label{table_real1}
\begin{tabular}{l|ccc} 
\toprule
\multicolumn{1}{l|}{Algorithm }          & \greg\ & \freg\ & \noreg\         \\ 
\midrule
AltMin-CG    &  0.2233   & 0.2233 &   1.4526     \\
ADMM  & 0.2193 & 0.2234 & 1.1205     \\
\bottomrule
\end{tabular}
\end{table}

\begin{figure}[h]
\centering
\setcounter{subfigure}{0}
\subfigure[AltMin-CG]{
\begin{minipage}[t]{0.37\linewidth}
\centering
\includegraphics[width=1\textwidth]{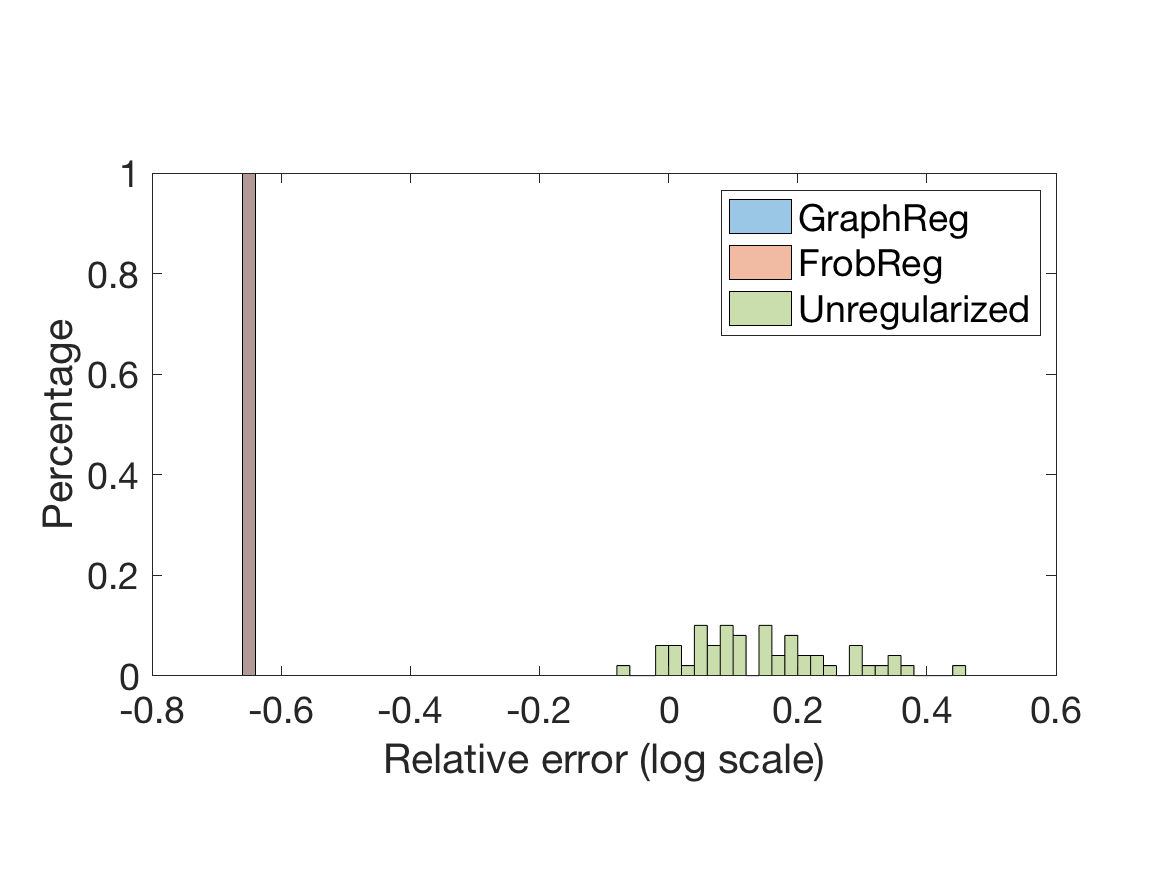}
\end{minipage}%
}%
\qquad 
\subfigure[ADMM]{
\begin{minipage}[t]{0.37\linewidth}
\centering
\includegraphics[width=1\textwidth]{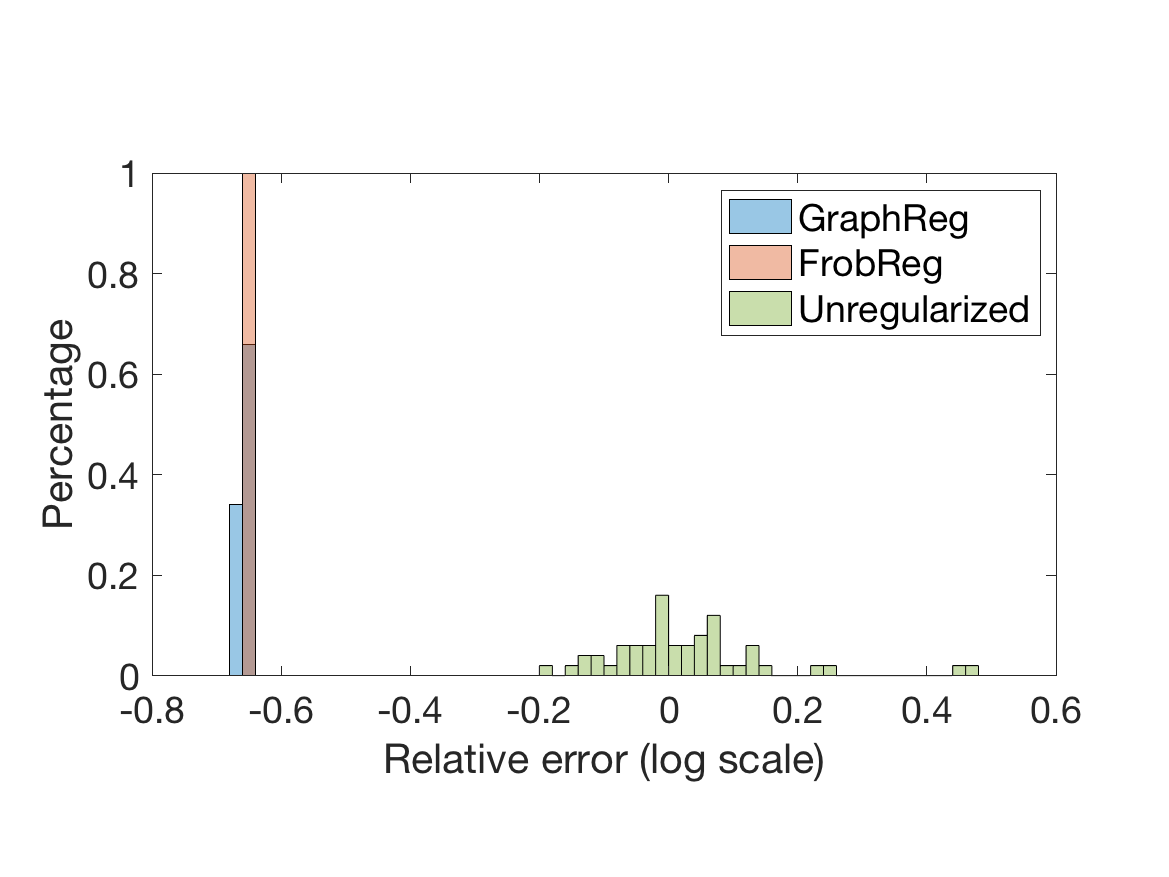}
\end{minipage}%
}%
\caption{Histogram of recovery errors of the three tensor completion models on the MovieLens dataset.}
\label{fig_real1}
\end{figure}

\paragraph*{\bf \revs{Experiment on the FIA dataset}}
For the graph Laplacian regularization, 
we construct the adjacency matrices following the ideas in \cite[Section
4.1]{narita2012tensor}.\footnote{\rev{A comparison between the proposed
methods and the TFAI method of \cite{narita2012tensor} is made; see 
\Cref{ssec:exp-time}.}} For $12$ chemical substances, we build the adjacency matrix with edge weights defined from the inverse of the Euclidean distance of the corresponding pairs of feature vectors. \rev{Along the index set of wavelengths and reaction times, respectively, we construct an adjacency matrix using the chain graph model,  
which models the connectivity of neighboring wavelengths or time stamps, since the observations on the chemical substances generally vary smoothly in the domain of wavelengths and also in time.} %

The sampling rate varies %
among $1\%$, $5\%$ and $7\%$. We present in Table \ref{table_real2} average relative errors after running $50$ random tests, and in Figure \ref{fig_real2_a} the histogram of relative errors under the sampling rate $1\%$. %
\revv{Similar to the comparative results on synthetic data, these results show that the solutions of the \greg\ model have smaller recovery errors compared to those of the other two models.}

\begin{table}[htpb]
\centering
\caption{Relative error of the three models on the FIA dataset: \greg\ (with graph Laplacian), \freg\ (without graph Laplacian), \noreg\ (no regularizer).}
\label{table_real2}
\begin{tabular}{cc|ccc}
 \toprule
 Sampling rate       &     {Algorithm}       & \greg\ & \freg\ & \noreg\ \\ \midrule
 \multirow{2}*{$1\%$}          & \multicolumn{1}{c}{AltMin-CG} &  0.1012   &  1.1096    &   6.6597         \\ \cmidrule(lr){2-5} 
    & \multicolumn{1}{c}{ADMM} &  0.2396	& 1.1385  & 1.5985		            \\  \midrule
 \multirow{2}*{$5\%$}                & \multicolumn{1}{c}{AltMin-CG} &  0.0146 & 0.3470   & 3.9619      \\ \cmidrule(lr){2-5}  
    & \multicolumn{1}{c}{ADMM} &   0.0209  & 0.2131   & 0.2131   \\ \midrule
 \multirow{2}*{$7\%$}                & \multicolumn{1}{c}{AltMin-CG} &  0.0126  & 0.2087   & 0.3837        \\ \cmidrule(lr){2-5}  
       & \multicolumn{1}{c}{ADMM} & 0.0180 &  0.0572  & 0.0572   \\   \bottomrule
\end{tabular}
\end{table}

\begin{figure}[H]
\centering
\setcounter{subfigure}{0}
\subfigure[AltMin-CG]{
\begin{minipage}[t]{0.4\linewidth}
\centering
\includegraphics[width=1\textwidth]{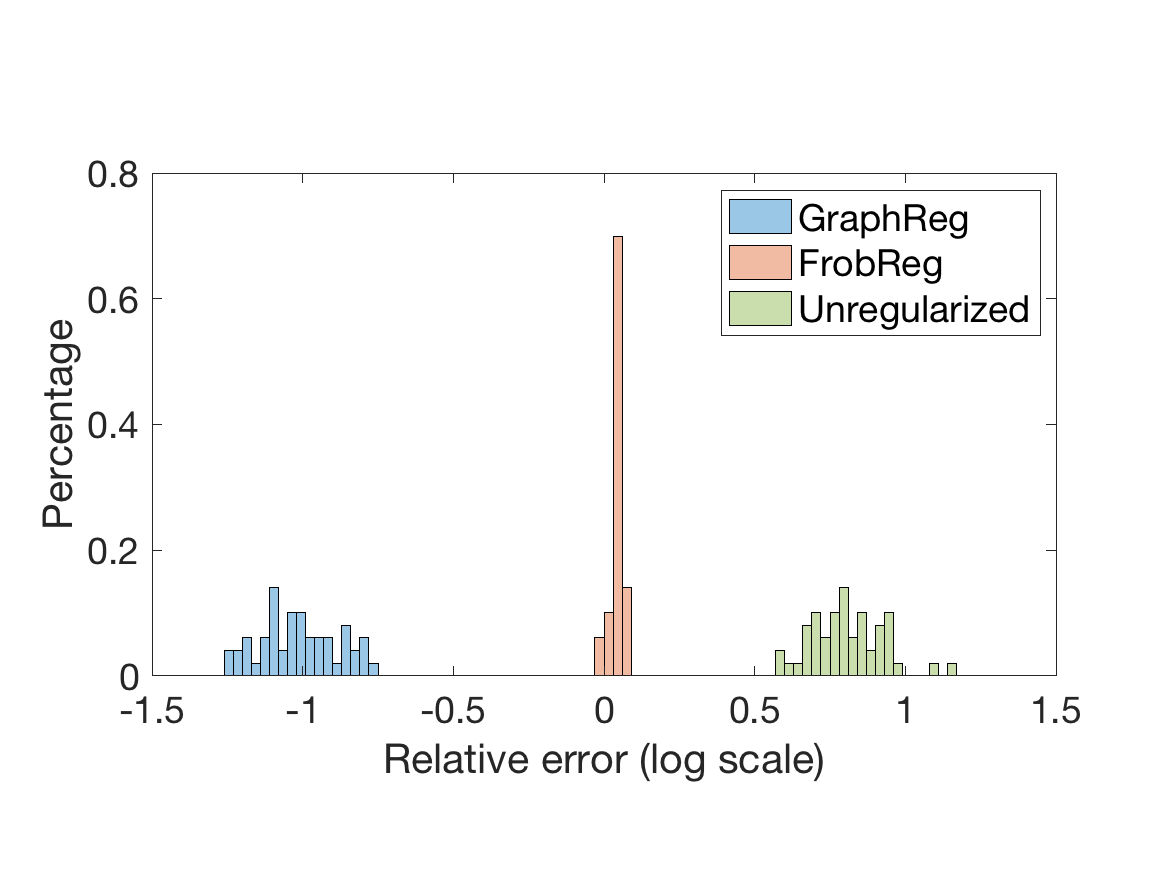}
\end{minipage}%
}%
\qquad 
\subfigure[ADMM]{
\begin{minipage}[t]{0.4\linewidth}
\centering
\includegraphics[width=1\textwidth]{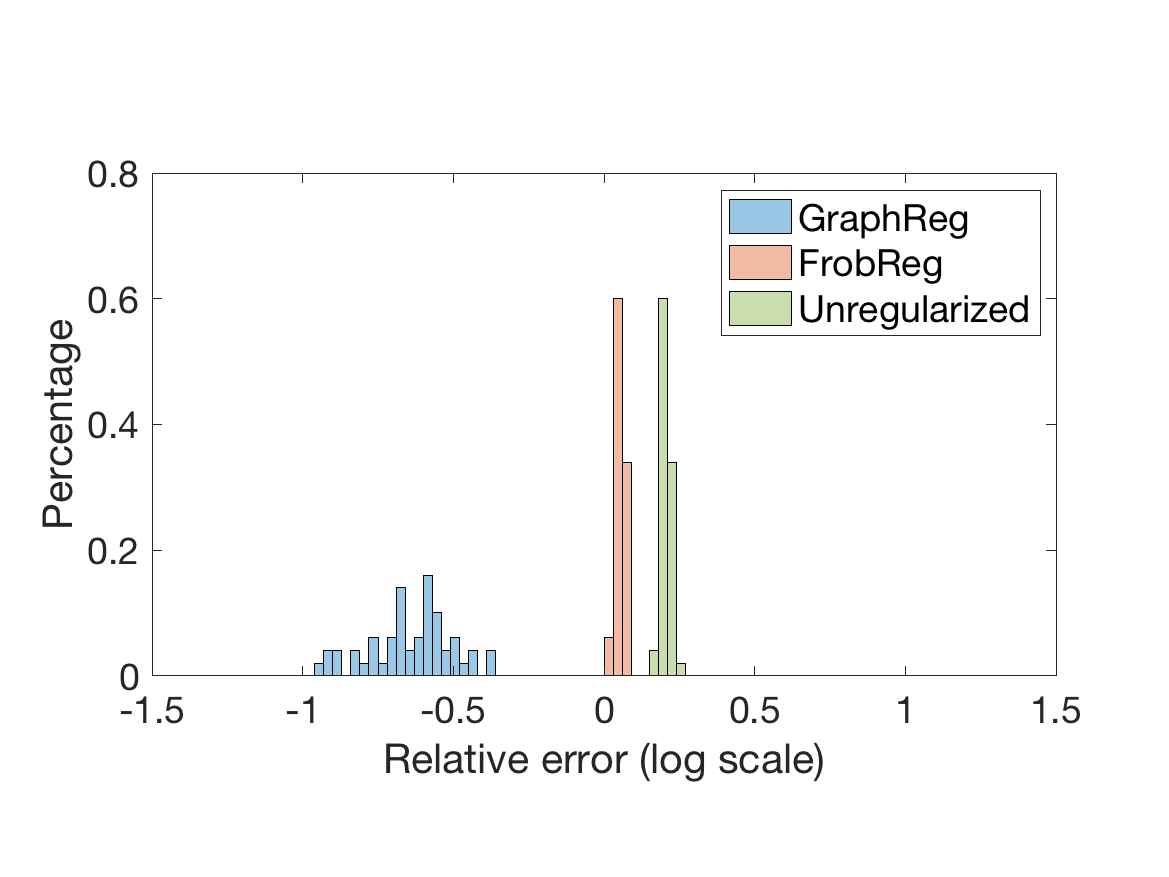}
\end{minipage}%
}%
\caption{Histogram of recovery errors of the three tensor completion models on the FIA dataset. The sampling rate is $1\%$.}
\label{fig_real2_a}
\end{figure}

\paragraph*{\bf Observations and results} 
\rev{
From the results on the real-world datasets, the graph-regularized
model~\eqref{problem} (\greg) shows improvements in the recovery accuracy on
both the MovieLens and the FIA datasets. %
Especially with the FIA dataset, the improvements by \greg\ is 
increasingly significant when the sampling rate decreases from 7\% to 1\%. %
A possible explanation is that the graphs constructed for the FIA dataset
are good enough in capturing the entrywise similarities in the (partially
observed) tensor while in the case of MovieLens, the graph constructed (using the
model \eqref{eq:epsNN-Gaussian}) is less adapted to the real-world user/movie
similarities.} 

\rev{Similarly and even more evidently, the results (\Cref{table_syn2}) on synthetic datasets (for which the graph information is given) show that the graph-regularized model~\eqref{problem} %
entails significant improvements over the other two models without
graph regularization, under all sampling rates tested. 
This observation validates the improvements by the \greg\ model for tensor completion %
especially when the proportion of the revealed entries (the sampling rate) is low. 
}

\subsection{Time efficiency evaluations}
\label{ssec:exp-time}

In this subsection, we focus on evaluating the time efficiency of our proposed algorithms under the same experimental settings described in the last subsection. 
In each of the following experiments, iteration information for our proposed algorithms is recorded. At each iteration, the recovery quality of the current iterate is evaluated in terms of the RMSE on test entries.

Besides our proposed algorithms, the state-of-art methods mentioned in \Cref{ssec:priorwork} are also tested using implementations that are publicly available or made available to us: 
(i) INDAFAC, the damped Gauss-Newton method proposed by Tomasi and
Bro~\cite{tomasi2005parafac}; 
(ii) CP-WOPT, a CP decomposition algorithm by Acar et al~\cite{acar2011scalable} 
which is available in the Tensor Toolbox~\cite{bader2012matlab}; 
(iii) \revv{BPTF, a Bayesian probabilistic tensor CPD
algorithm~\cite{xiong2010temporal};} 
(iv) TFAI, an \revv{algorithm for tensor factorization with the within-mode auxiliary information;} 
(v) TNCP, an ADMM algorithm for solving a matrix trace-norm regularized model~\cite{liu2014trace}; and 
(vi) AirCP, an ADMM algorithm for solving the
CP-based tensor completion using auxiliary graph information~\cite{ge2016uncovering}. 
The parameters involved in the models of TFAI, TNCP and AirCP are chosen after cross validation.

\paragraph*{\bf Experiment on synthetic data} 

\revv{Under the same experimental settings as for Table~\ref{table_syn2}, the RMSEs
of the iterates given by the tested methods %
are shown in Figure~\ref{fig_experiment6_a}.  
The iterative results in the figures are taken from one test
randomly chosen from the repeated tests.} 
\revv{Figure~\ref{fig_experiment6_a} shows the results under the sampling rates 
$0.3\%$ and $0.5\%$, in which we observe that proposed algorithms, AltMin-CG and ADMM, have a better time efficiency than the rest of the tested methods. Detailed observations on these results are given in Section~\ref{ssec:exp-obs}.}

\begin{figure}[htbp]
\centering
\subfigure[Sampling rate $0.3\%$]{
\includegraphics[width=.70\textwidth]{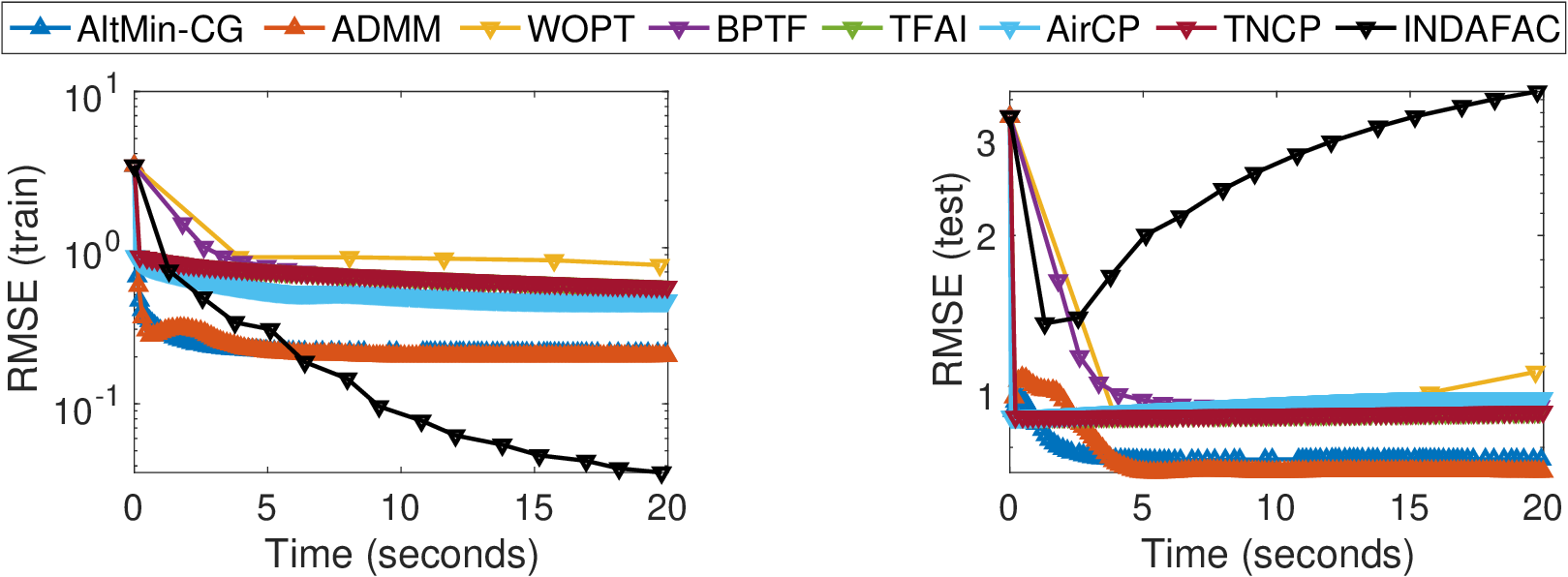}
}
\\
\subfigure[Sampling rate $0.5\%$]{
\includegraphics[width=.70\textwidth]{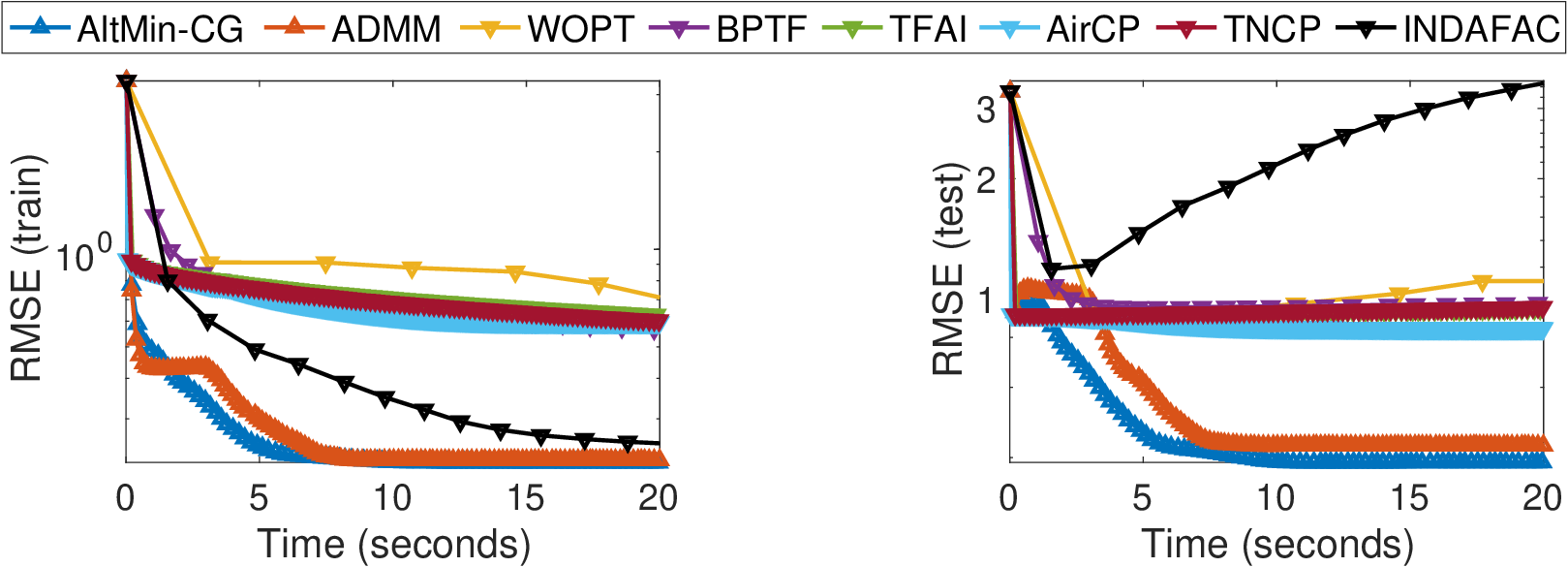}
} 
\caption{Iterative results of the tested algorithms on synthetic data: RMSE on training and test sets by accumulative time per-iteration.  The sampling rates are $0.3\%$ and $0.5\%$.}
\label{fig_experiment6_a}
\end{figure}

\paragraph*{\bf Experiment on MovieLens} 
\revv{Under the same experimental settings on the MovieLens dataset and with the same graph construction method as in the previous subsection, %
we show the RMSEs of the iterates given by the tested methods in Figure~\ref{fig_experiment7}.
The iterative results in the figure are taken from one test
randomly chosen from the repeated tests.} 
In particular, the labels ``AltMin-CG1'' and ``ADMM1'' represent the results of these algorithms under the graph-agnostic, nuclear norm-based model (NuclReg-TC). 
The results therein shows that the AltMin-CG and ADMM methods remain competitive on this dataset, in particular their time efficiency is comparable to AirCP.

\begin{figure}[!htbp]
\centering
\includegraphics[width=.72\textwidth]{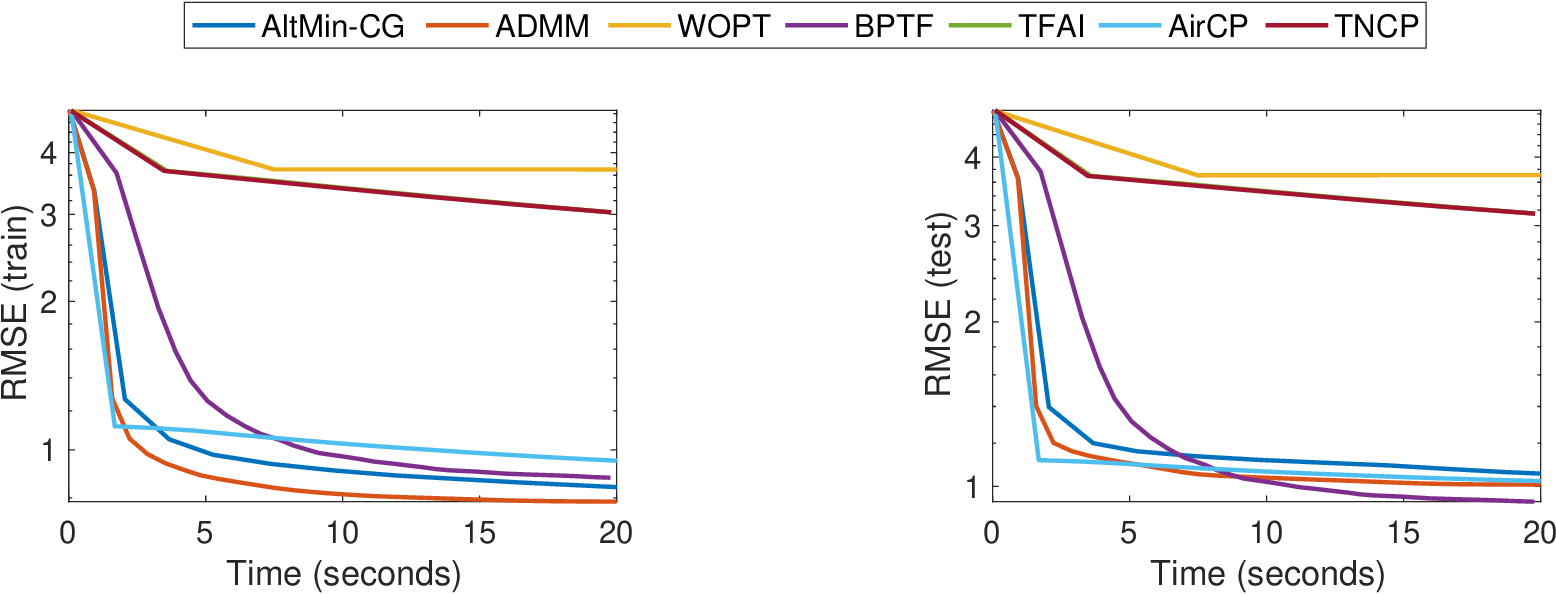}
\caption{Iterative results of the tested algorithms on the MovieLens dataset: RMSE on training and test sets by accumulative time per-iteration.}
\label{fig_experiment7}
\end{figure}

\begin{figure}[!htbp]
\centering
\subfigure[Sampling rate $1.0\%$]{
\includegraphics[width=.72\textwidth]{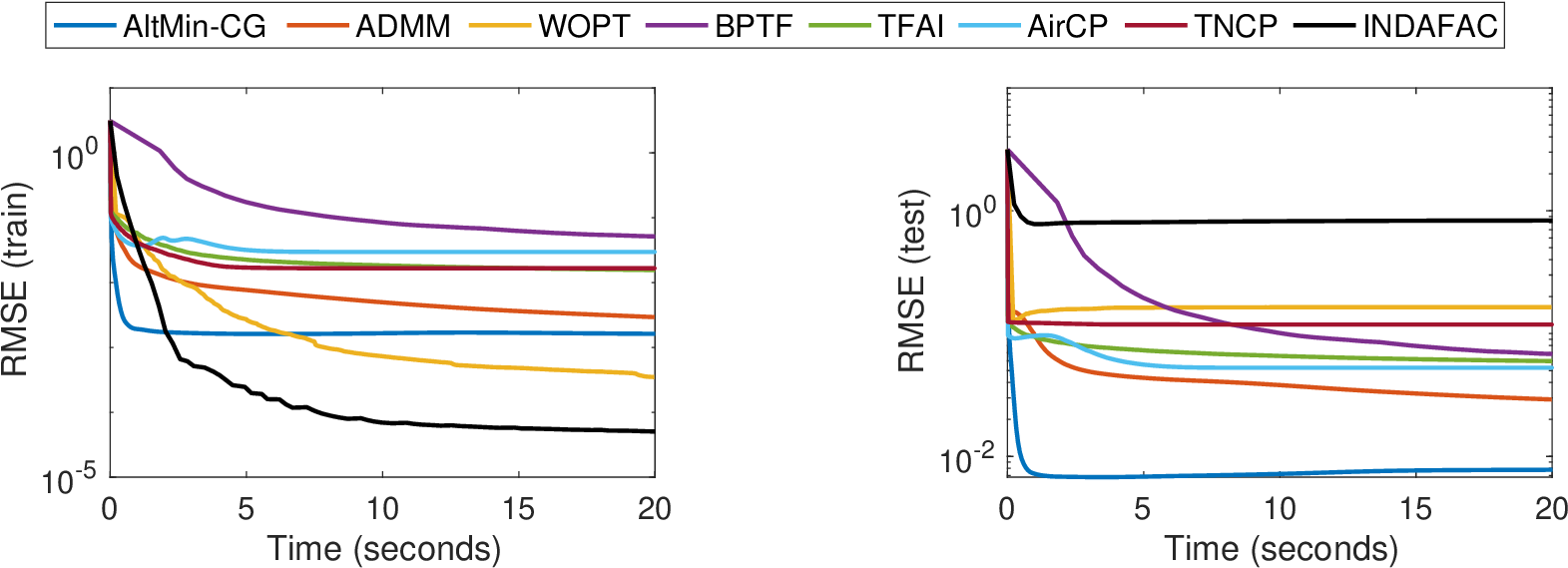}
}
\\
\subfigure[Sampling rate $5.0\%$]{
\includegraphics[width=.72\textwidth]{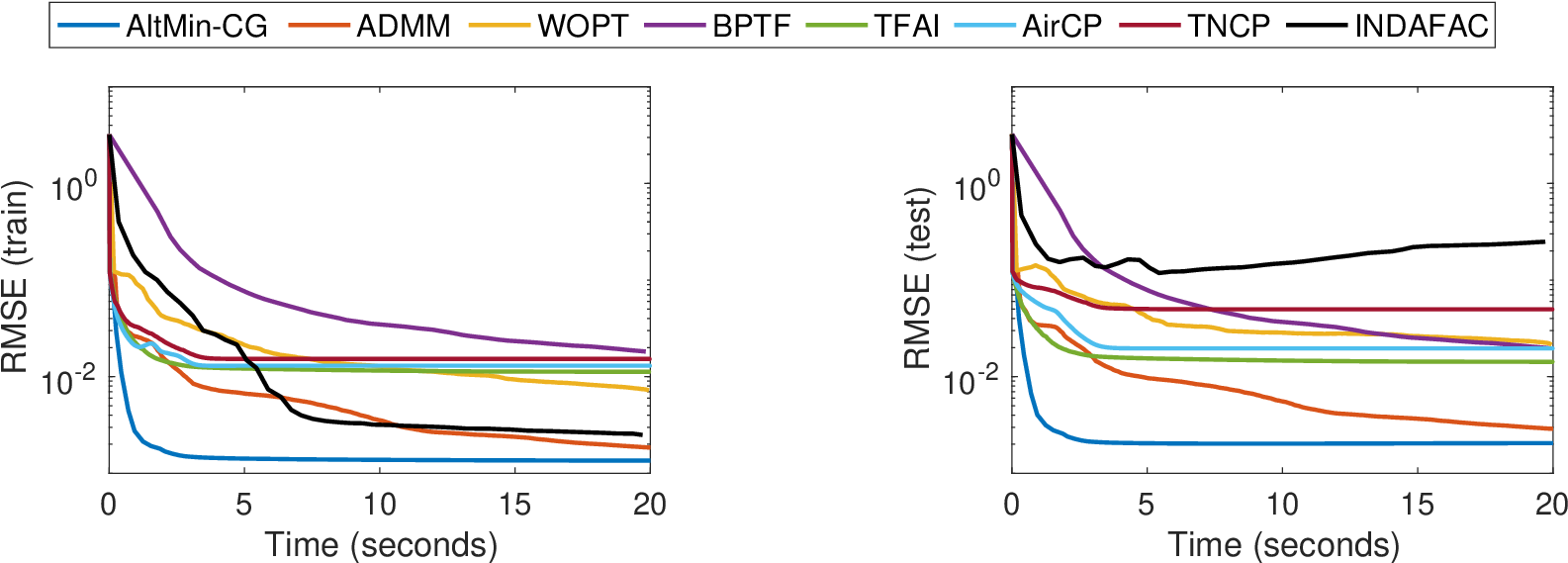}
} 
\caption{Iterative results of the tested algorithms on the FIA dataset: RMSE on training and test sets by accumulative time per-iteration.}
\label{fig_experiment8_a}
\end{figure}

\paragraph*{\bf Experiment on the FIA dataset} 
\revv{Under the same experimental settings on the FIA dataset and with the same graph construction method as in the previous subsection, %
we show the RMSEs of the iterates given by the tested methods in Figure~\ref{fig_experiment8_a}. 
The iterative results in the figure are taken from one test
randomly chosen from the repeated tests.} 
These results show that the proposed algorithms AltMin-CG and ADMM yield the best recovery performances (in terms of test RMSE).

\paragraph*{\bf Observations and results}\label{ssec:exp-obs} 
From Figure~\ref{fig_experiment6_a}\comm{~and Figure~\ref{fig_experiment6_b}}, we
observe that most of the methods perform comparably in time efficiency. In
particular, the proposed AltMin-CG and ADMM methods, \revf{followed by} AirCP,
TFAI, TNCP, are the fastest in terms of training RMSE. \revf{Also, AltMin-CG and ADMM
achieve the lowest test RMSEs, under low sampling rates.} %
This result is encouraging since it shows that our methods and model for tensor
completion can achieve robust recovery where only a small fraction of data is
available.

For real data in \revf{Figure}~\ref{fig_experiment7}, AirCP, TFAI, AltMin-CG
and ADMM \revf{have comparable time efficiencies and are the fastest among all tested
algorithms in terms of training error}. BPTF is slower than the aforementioned
methods but obtains the lowest test RMSE. In Figure~\ref{fig_experiment8_a}\comm{~and
Figure~\ref{fig_experiment8_b}}, \revf{we also observe} that the proposed AltMin-CG and ADMM methods
outperform most other methods both in efficiency and accuracy with \revf{the given
training data that correspond to low sampling rates.} 

\section{Conclusion}\label{sec:conclusion}
\rev{In this paper, we studied a CP decomposition-based tensor completion model~\eqref{problem} 
that involves graph regularization. 
This model is motivated by two main reasons: (i) the CP decomposition enables a memory-efficient model for
low-rank tensors and (ii) the use of the graph Laplacian-based regularizer
is shown to be effective in many tasks such as semi-supervised learning, image
restoration and also matrix completion, which makes graph regularization a tempting tool for tensor completion. 
For the optimization of this completion model of $k$-th order tensors, we proposed an alternating minimization algorithm and an ADMM algorithm
adapted to the block structure of the underlying problem. 
Within the alternating minimization procedure, we \revn{showed} that each of the $k$ subproblems 
is a quadratic minimization problem, and we \revn{adapted} a linear CG method for
solving these subproblems. An efficient Hessian-vector multiplication \revn{was} 
used for the linear CG subroutine. Besides, we proposed an 
ADMM algorithm, which further
resolves an undesirable coupling effect of graph regularization on the
optimization of the problem. The convergence property of the proposed AltMin
algorithm \revn{was} analysed.}  

\rev{From the results of various numerical experiments, we verified that the
graph-regularized tensor completion model (\ref{problem}) produces improved tensor
completion
results \revn{with respect to} tensor completion models without graph regularization. This observation is especially significant when the sample rate is
small. The proposed algorithms, AltMin-CG (Algorithm~\ref{alg:Altmin_CG}) and ADMM
(Algorithm~\ref{algorithm_ADMM}), have also shown to have superior or comparable
completion results with improved time efficiency at the same time, compared
to state-of-art methods.  
}

\section*{Acknowlegement} 
This work was supported by the Fonds de la Recherche Scientifique -- FNRS and
the Fonds Wetenschappelijk Onderzoek -- Vlaanderen under EOS Project no.\
30468160 (SeLMA -- Structured low-rank matrix/tensor approximation: numerical optimization-based algorithms and applications) \revn{and by the Fonds de la Recherche Scientifique -- FNRS under Grant no.\ T.0001.23}. The second author was supported by the FNRS through a FRIA scholarship. The third author was supported by the Young Elite Scientist Sponsorship Program by CAST.

\bibliographystyle{siam}
\bibliography{Tensor_ref02}

\end{document}